\documentclass{birkjour}

\usepackage[frenchb,english]{babel}

\usepackage{amsmath}
\usepackage{amssymb}
\usepackage{amsfonts}
\usepackage{enumerate}
\usepackage{vmargin}
\usepackage[all]{xy}
\usepackage{mathtools}

\setmarginsrb{3cm}{3cm}{2.5cm}{2cm}{0.5cm}{1cm}{1.5cm}{4cm}

\newcommand{\A}{\ensuremath{\mathcal{A}}}

\newcommand{\Mat}{\ensuremath{\mathbb{M}}}

\newcommand{\R}{\ensuremath{\mathbb{R}}}
\newcommand{\C}{\ensuremath{\mathbb{C}}}
\newcommand{\T}{\ensuremath{\mathbb{T}}}

\newcommand{\tr}{\ensuremath{\mathop{\rm Tr\,}\nolimits}}
\newcommand{\clos}{\ensuremath{\mathop{\rm closure\,}\nolimits}}

\renewcommand{\leq}{\ensuremath{\leqslant}}
\renewcommand{\geq}{\ensuremath{\geqslant}}

\renewcommand{\l}{\lambda}

\renewcommand{\O}{{\Omega}}
\renewcommand{\o}{{\omega}}

\newcommand\be{\begin{align}}
\newcommand\ee{\end{align}}
\newcommand{\norm}[1]{ \| #1  \|}
\newcommand{\bnorm}[1]{ \big\| #1  \big\|}
\newcommand{\Bnorm}[1]{ \Big\| #1  \Big\|}
\newcommand{\bgnorm}[1]{ \bigg\| #1  \bigg\|}
\newcommand{\Bgnorm}[1]{ \Bigg\| #1  \Bigg\|}
\newcommand{\xra}{\xrightarrow}
\newcommand{\otpb}{\ot_{\gamma}}

\newcommand{\otp}{\widehat{\ot}}
\newcommand{\ot}{\otimes}
\def\Im{{\rm Im} \, }
\def\Ker{{\rm Ker} \, }
\def\min{{\rm min} }
\def\fin{{\rm fin} }
\def\card{{\rm card} }

 \newtheorem{thm}{Theorem}[section]
 
 \newtheorem{lemma}[thm]{Lemma}
 \newtheorem{prop}[thm]{Proposition}
 \theoremstyle{definition}
 
 \theoremstyle{remark}
 \newtheorem{remark}[thm]{Remark}
 
 \numberwithin{equation}{section}

\begin{document}

\title[Noncommutative Fig\`a-Talamanca-Herz algebras]
 {Noncommutative Fig\`a-Talamanca-Herz\\ algebras for Schur multipliers}

\date{}
\author{C\'edric Arhancet }
\address{%
Laboratoire de Math\'ematiques\\
Universit\'e de Franche-Comt\'e\\
25030 Besan\c{c}on Cedex\\
France}

\email{cedric.arhancet@univ-fcomte.fr}

\thanks{ This work is partially supported by ANR 06-BLAN-0015.}



\date{January 1, 2011}

\begin{abstract}
In this work, we introduce a noncommutative analogue of the
Fig\`a-Talamanca-Herz algebra $A_p(G)$ on the natural predual of the
operator space $\frak{M}_{p,cb}$ of completely bounded Schur
multipliers on the Schatten space $S_p$. We determine the isometric
Schur multipliers and prove that the space $\frak{M}_{p}$ of bounded
Schur multipliers on the Schatten space $S_p$ is the closure in the
weak operator topology of the span of isometric multipliers.
\end{abstract}

\maketitle


\section{Introduction}


The Fourier algebra $A(G)$ of a locally compact group $G$ was
introduced by P. Eymard in \cite{Eym}. The algebra $A(G)$ is the
predual of the group von Neumann algebra $VN(G)$. If $G$ is abelian
with dual group $\widehat{G}$, then the fourier transform induces an
isometric isomorphism of $L_1\big(\widehat{G}\big)$ onto $A(G)$. In
\cite{Fig}, A. Fig\`a-Talamanca showed, if $G$ is abelian, that the
natural predual of the Banach space of the bounded Fourier
multipliers on $L_p(G)$ is isometrically isomorphic to a space
$A_p(G)$ of continuous functions on $G$. Moreover $A_2(G)=A(G)$
isometrically. In \cite{Her} and \cite{Eym}, C. Herz proved that the
space $A_p(G)$ is a Banach algebra for the usual product of
functions (see also \cite{Pie}). Hence $A_p(G)$ is an $L_p$-analogue
of the Fourier algebra $A(G)$. These algebras are called
Fig\`a-Talamanca-Herz algebras. In \cite{Run1}, V. Runde introduced
an operator space analogue $OA_p(G)$ of the algebra $A_p(G)$. The
underlying Banach space of $OA_p(G)$ is different from the Banach
space $A_p(G)$. Moreover, it is possible to show (in using a
suitable variant of \cite[Theorem 5.6.1]{Lar}) that $OA_p(G)$ is the
natural predual of the operator space of the completely bounded
Fourier multipliers. We refer to \cite{Daw1}, \cite{Daw2},
\cite{Lam} and \cite{Run2} for other operator space analogues of
$A_p(G)$.

The purpose of this article is to introduce noncommutative analogues
of these algebras in the context of completely bounded Schur
multipliers on Schatten spaces $S_p$. Recall that a map $T \colon
S_p \to S_p$ is completely bounded if $Id_{S_p} \ot T$ is bounded on
$S_p(S_p)$. If $1 \leq p < \infty$, the operator space $CB(S_p)$ of
completely bounded maps from $S_p$ into itself is naturally a dual
operator space. Indeed, we have a completely isometric isomorphism
$CB(S_p)=\big(S_p \otp S_{p^*}\big)^*$ where $\otp$ denote the
operator space projective tensor product. Moreover, we will prove
that the subspace $\frak{M}_{p,cb}$ of completely bounded Schur
multipliers is a maximal commutative subset of $CB(S_p)$.
Consequently, the subspace $\frak{M}_{p,cb}$ is w*-closed in
$CB(S_p)$. Hence $\frak{M}_{p,cb}$ is naturally a dual operator
space with $\frak{M}_{p,cb}=\big(S_p \otp
S_{p^*}/(\frak{M}_{p,cb})_{\perp}\big)^*$. If we denote by
$\psi_p\colon S_p \otp S_{p^*}\to S_1$ the map $(A,B) \mapsto A*B$,
where $*$ is the Schur product, we will show that
$(\frak{M}_{p,cb})_{\perp}=\Ker \psi_p$. Now, we define the operator
space $\frak{R}_{p,cb}$ as the space $\Im \psi_p$ equipped with the
operator space structure of $S_p \otp S_{p^*}/\Ker \psi_p$. We have
completely isometrically
$\big(\frak{R}_{p,cb}\big)^*=\frak{M}_{p,cb}$. Moreover, by
definition, we have a completely contractive inclusion
$\frak{R}_{p,cb} \subset S_1$. Recall that elements of $S_1$ can be
regarded as infinite matrices. Our principal result is the following
theorem.
\begin{thm}
Suppose $1 \leq p < \infty$. The predual $\frak{R}_{p,cb}$ of the
operator space $\frak{M}_{p,cb}$ equipped with the usual matricial
product or the Schur product is a completely contractive Banach
algebra.
\end{thm}
In \cite{Str} and \cite{Par}, R. S. Strichartz and S. K. Parott
showed that if $1 \leq p \leq\infty$, $p\not=2$ every isometric
Fourier multiplier on $L_p(G)$ is a scalar multiple of an operator
induced by a translation. In \cite{Fig}, A. Fig\`a-Talamanca showed
that the space of bounded Fourier multipliers is the closure in the
weak operator topology of the span of these operators. We give
noncommutative analogues of these two results.
\begin{thm}
\begin{enumerate}
  \item  Suppose $1 \leq p \leq\infty $. If $p \not=2$, an isometric Schur multiplier on $S_p$ is defined by a matrix $[a_ib_j]$ with
        $a_i,b_j \in \mathbb{T}$.
  \item  Suppose $1 \leq p < \infty $. The space $\frak{M}_{p}$ of bounded Schur multipliers on $S_p$ is the closure of the span of isometric Schur
         multipliers in the weak operator topology.
\end{enumerate}
\end{thm}
The paper is organized as follows.

In \S2, we fix notations and we show that the natural preduals of
$\frak{M}_{p}$ and $\frak{M}_{p,cb}$ admit concrete realizations as spaces of matrices. We
give elementary properties of these spaces.

In \S3, we show that the operator space $\frak{R}_{p,cb}$ equipped
with the matricial product is a completely contractive Banach
algebra.

In \S4, we turn to the Schur product. We observe that the natural
predual $\frak{R}_p$ of the Banach space $\frak{M}_{p}$ of bounded
Schur multipliers is a Banach algebra for the Schur product.
Moreover, we show that the space $\frak{R}_{p,cb}$ equipped with the
Schur product is a completely contractive Banach algebra.

In \S5, we determine the isometric Schur multipliers on $S_p$ and
prove that the space $\frak{M}_{p}$ is the closure in the weak
operator topology of the span of isometric multipliers.


\section{Preduals of spaces of Schur multipliers}

Let us recall some basic notations. Let $\T=\{z \in \C\ |\ |z|=1 \}$
and $\delta_{ij}$ the symbol of Kronecker.

If $E$ and $F$ are Banach spaces, $B(E,F)$ is the space of bounded
linear maps between $E$ and $F$.  We denote by $\otpb$ the Banach
projective tensor product. If $E,F$ and $G$ are Banach spaces we
have $(E\otpb F)^*=B(E,F^*) $ isometrically. In particular, if $E$
is a dual Banach space, $B(E)$ is also a dual Banach space. If
$(E_0,E_1)$ is a compatible couple of Banach spaces we denote by
$(E_0,E_1)_\theta$ the intermediate space obtained by complex
interpolation between $E_0$ and $E_1$.

The readers are refereed to \cite{BLM}, \cite{ER1}, \cite{Pau} and
\cite{Pis4} for the details on operator spaces and completely
bounded maps. We let $CB(E,F)$ for the space of all completely
bounded maps endowed with the norm
$$
\norm{T}_{E\xra{}F,cb}=\sup_{n\geq 1} \bnorm{Id_{M_n} \ot
u}_{M_n(E)\xra{}M_n(F)}.
$$
When $E$ and $F$ are two operator spaces,
$CB(E,F)$ is an operator space for the structure corresponding to
the isometric identifications
$M_n\big(CB(E,F)\big)=CB\big(E,M_n(F)\big)$. The dual operator space
of $E$ is $E^*=CB(E,\C)$. If $E$ and $F$ are operator spaces then
the adjoint map $T\mapsto T^*$ from $CB(E,F)$ into $CB(F^*,E^*)$ is
a complete isometry.

If $I$ is a set, we denote by $C_{I}$ the operator space
$B\big(\C,\ell_2^I\big)$ and by $R_{I}$ the operator space
$B\big(\overline{\ell_2^I},\C\big)$. We have a complete isometry
$B\big(\ell_2^I\big)=CB\big(C_{I}\big)$ (see \cite[(1.14)]{BLM}).

The complex interpolated space between two compatible operator
spaces $E_0$ and $E_1$ is the usual Banach space $E_\theta$ with the
matrix norms corresponding to the isometric identifications
$M_n(E_\theta)=\big(M_n(E_0),M_n(E_1)\big)_{\theta}$. Let $F_0,F_1$
be two other compatible operator spaces. Let $\varphi\colon E_0+E_1
\to F_0+F_1$ be a linear map. If $\varphi$ is completely bounded as
a map from $E_0$ into $F_0$, and from $E_1$ into $F_1$, then, for
any $0\leq \theta\leq 1$, $\varphi$ is completely bounded from
$E_\theta$ into $F_\theta$ with
\begin{equation*}
\norm{\varphi}_{cb,E_\theta\xra{}F_\theta} \leq
\big(\norm{\varphi}_{cb,E_0\xra{}F_0}\big)^{1-\theta}\big(\norm{\varphi}_{cb,E_1\xra{}F_1}\big)^{\theta}.
\end{equation*}
If $E_0 \cap E_1$ is dense in both $E_0$ and $E_1$, we have a
completely contractive inclusion
$$
\big(CB(E_0),CB(E_1)\big)_\theta \subset CB(E_\theta)
$$
(see \cite[Lemma 0.2]{Har}).

We denote by $\otp$ the operator space projective tensor product, by
$\ot_{\min}$ the operator space minimal tensor product, by $\ot_h$
the Haagerup tensor product, by $\ot_{\sigma h}$ the normal Haagerup
tensor product, by $\overline{\ot}$ the normal spatial tensor
product, by $\ot_{w^*h}$ the weak* Haagerup tensor product and by
$\ot_{eh}$ the extended Haagerup tensor product (see \cite{BLM},
\cite{ER2} and \cite{Spr}). Suppose that $E,F,G$ and $H$ are
operator spaces. If $\varphi\colon E\to F$ and $\psi\colon G\to H$
are completely bounded maps then the maps $\varphi\ot \psi\colon E
\ot_h G\to F \ot_h H$ and $\varphi\ot \psi\colon E \otp G \to F \otp
H$ are completely bounded and we have
\begin{equation*}
\norm{\varphi\ot \psi}_{cb,E \ot_h G\xra{}F \ot_h H}\leq
\norm{\varphi}_{cb,E\to F}\norm{\psi}_{cb,G\to H}
\end{equation*}
and
\begin{equation*}
\norm{\varphi\ot \psi}_{cb,E \otp G\xra{}F \otp H}\leq
\norm{\varphi}_{cb,E\to F}\norm{\psi}_{cb,G\to H}.
\end{equation*}

If $E,F$ are operator spaces, we have $E\ot_h F \subset E
\ot_{w^*h}F$ completely isometrically (see \cite{BLM} page 43).

If $E,F$ and $G$ are operator spaces, we denote by $CB(E\times F,G)$
the space of jointly completely bounded map. We have
$$
CB(E\times F,G)=CB\big(E\otp F,G\big)=CB\big(E,CB(F,G)\big)
$$
completely isometrically. Consequently, we have $\big(E\otp
F\big)^*=CB(E,F^*)$ completely isometrically. In particular, if $E$
is a dual operator space, $CB(E)$ is also a dual operator space.

At several times, we will use the next easy lemma left to the
reader.
\begin{lemma}
\label{lemmaVW} Suppose $E$ and $F$ are operator spaces. Let
$V\colon E\to F$ and $W\colon F \to E$ be any completely contractive
maps. Then the map
$$
\begin{array}{cccc}
  \Theta_{V,W}:  &  CB(E)   &  \longrightarrow   & CB(F)   \\
    &  T    &  \longmapsto       &  VTW  \\
\end{array}
$$
is completely contractive. Moreover, if $E$ and $F$ are reflexive
then this map is also w*-continuous.
\end{lemma}

A Banach algebra $\A$ equipped with an operator space structure is
called completely contractive if the algebra product $(a,b)\xra{}ab$
from $\A \times \A$ to $\A$ is a jointly completely contractive
bilinear map.

We equip $\ell_\infty^{I}$ with its natural operator space structure
coming from its structure as a $C^*$-algebra and the Banach space
$\ell_1^I$ with its natural operator space structure coming from its
structure of predual of $\ell_\infty^I$.

If $I$ is an index set and if $E$ is a vector space, we write
$\Mat_I(E)$ for the space of the $I \times I$ matrices with entries
in $E$. We denote by $\Mat^{\fin}_{I}(E)$ the subspace of matrices
with a finite number of non null entries. For $I={\{1,\ldots,n\}}$,
we simplify the notations, we let $M_n(E)$ for
$\Mat_{\{1,\ldots,n\}}(E)$. We write $\Mat_{\fin}$ for
$\Mat_{\mathbb{N}}^{\fin}(\mathbb{C})$. We use the inclusion $\Mat_I
\ot \Mat_I \subset \Mat_{I \times I}$ with the identification $[A
\ot B]_{(t,r),(u,s)}=a_{tu}b_{rs}$. For all $i,j,k,l\in I$, the
tensor $e_{ij} \ot e_{kl}$ identifies to the matrix
$[\delta_{it}\delta_{ju}\delta_{kr}\delta_{ls}]_{(t,r),(u,s) \in
I\times I} $ (see \cite{ER1} page 5 for more information on these
identifications).

Given a set $I$, the set $\mathcal{P}_f(I)$ of all finite subsets of
$I$ is directed with respect to set inclusion. For $J\in
\mathcal{P}_f(I)$ and $A\in \Mat_{I}$, we write $\mathcal{T}_{J}(A)$
for the matrix obtained from $A$ by setting each entry to zero if
its row and column index are not both in $J$. We call
$\big(\mathcal{T}_{J}(A)\big)_{J \in \mathcal{P}_f(I)}$ the net of
finite submatrices of $A$.

The Schatten-von Neumann class $S_p^{I}$, $1 \leq p< \infty$, is the
space of those compact operators $A$ from $\ell_2^I$ into $\ell_2^I$
such that
$\norm{A}_{S_p^{I}}=\big(\tr(A^*A)^{\frac{p}{2}}\big)^{\frac{1}{p}}<\infty$.
The space $S_\infty^{I}$ of compact operators from $\ell_2^I$ into
$\ell_2^I$ is equipped with the operator norm. For $I=\mathbb{N}$,
we simplify the notations, we let $S_p$ for $S_p^{\mathbb{N}}$. The
space $S_\infty^{I}\big(S_\infty^{K}\big)$ of compact operators from
$\ell_2^I \ot_{2} \ell_2^K$ into $\ell_2^I \ot_{2} \ell_2^K$ is
equipped with the operator norm. If $1 \leq p < \infty$, the space
$S_p^{I}\big(S_p^{K}\big)$ is the space of those compact operators
$C$ from $\ell_2^I \ot_{2} \ell_2^K$ into $\ell_2^I \ot_{2}
\ell_2^K$ such that $\norm{C}_{S_p^{I}(S_p^{K})} = \big((\tr \ot
\tr)(C^*C)^{\frac{p}{2}}\big)^{\frac{1}{p}} < \infty$.

Elements of $S_p^{I}$ are regarded as matrices $A=[a_{ij}]_{i,j \in
I}$ of $\Mat_{I}$. If $A \in S_p^{I}$ we denote by $A^T$ the
operator of $S_p^{I}$ whose the matrix is the matrix transpose of
$A$. If $1 \leq p \leq \infty$, $A \in S_p^{I}$ and $B \in
S_{p^*}^{I}$, the operator $AB^T$ belongs to $S_1^{I}$. We let
$\langle A,B\rangle_{S_p^{I},S_{p^*}^{I}}=\tr\big(AB^T\big)$. We
have $ \langle A,B\rangle_{S_p^{I},S_{p^*}^{I}}=\lim_{J}\sum_{i,j\in
J}a_{ij}b_{ij}$.

We equip $S_\infty^{I}$ with its natural operator space structure
coming from its structure as a $C^*$-algebra. We equip $S_1^{I}$
with its natural operator space structure coming from its structure
as dual of $S_\infty^{I}$. If $1 < p < \infty $, we give on
$S_p^{I}$ the operator space structure defined by $S_p^{I} =
\big(S_\infty^{I},S_1^{I}\big)_{\frac{1}{p}}$ completely
isometrically (see \cite{Pis4} page 140 for interesting remarks on
this definition). By the same way, we define an operator space
structure on $S_p^{I}\big(S_p^{K}\big)$. We have completely
isometrically
$S_p^{I}\big(S_p^{K}\big)=S_p^{K}\big(S_p^{I}\big)=S_p^{I\times K}$.
We will often silently use these identifications. By the same way,
we define $S_p^{I}\big(S_p^{K}(S_p^{L})\big)$ and similar operator
space structures. G. Pisier showed that a map $T\colon S_p^{I} \to
S_p^{I}$ is completely bounded if $Id_{S_p} \ot T$ is bounded on
$S_p\big(S_p^{I}\big)$ (see \cite[Lemma 1.7]{Pis2}). The readers are
refereed to \cite{Pis2} for the details on operator space structures
on the Schatten-von Neumann class.

We denote by $*$ the Schur (Hadamard) product: if
$A=[a_{ij}]_{i,j\in I}$ and $B=[b_{ij}]_{i,j\in I }$ are matrices of
$\Mat_I$ we have $A*B=[a_{ij}b_{ij}]_{i,j \in I}$. We recall that a
matrix $A$ of $\Mat_{I}$ defines a Schur multiplier $M_A$ on
$S_p^{I}$ if for any $B \in S_p^{I}$ the matrix $M_{A}(B)=A*B$
represents an element of $S_p^{I}$. In this case, by the closed
graph theorem, the linear map $B \mapsto M_A(B)$ is bounded on
$S_p^{I}$. The notation $\frak{M}_{p}^{I}$ stands for the algebra of
all bounded Schur multipliers on the Schatten space $S_p^{I}$. We
denote by $\frak{M}_{p,cb}^{I}$ the space of completely bounded
Schur multipliers on $S_p^{I}$. We give the space
$\frak{M}_{p,cb}^{I}$ the operator space structure induced by
$CB\big(S_p^{I}\big)$. For $I=\mathbb{N}$, we simplify the
notations, we let $\frak{M}_{p}$ for $\frak{M}_{p}^{\mathbb{N}}$ and
$\frak{M}_{p,cb}$ for $\frak{M}_{p,cb}^{\mathbb{N}}$. Recall that if
$A\in S_{p}^{I}$, we have $M_A\in \frak{M}_{p}^{I}$ (see \cite{BLM}
page 225).

If $M_{C} \in \frak{M}_{p}^{I}$, we have $M_{C} \in
\frak{M}_{p^*}^{I}$. Moreover, if $A\in S_p^{I}$ and $B\in
S_{p^*}^{I}$, we have
$$
\big\langle
M_{C}(A),B\big\rangle_{S_{p}^{I},S_{p^*}^{I}}=\big\langle
A,M_{C}(B)\big\rangle_{S_{p}^{I},S_{p^*}^{I}}.
$$
If $1 \leq p \leq \infty$, the Banach spaces $\frak{M}_{p}^{I}$ and
$\frak{M}_{p^*}^{I}$ are isometric and the operator spaces
$\frak{M}_{p,cb}^{I}$ and $\frak{M}_{p^*,cb}^{I}$ are completely
isometric. We have $\frak{M}_{\infty}^{I}=\frak{M}_{\infty,cb}^{I}$
isometrically (see e.g. \cite[Remark 2.2]{Neu} and \cite[Lemma
2]{Hla}). Moreover, we have $\frak{M}_{\infty,cb}^{I}=\ell_\infty^I
\ot_{w^*h} \ell_\infty^I$ completely isometrically (see e.g.
\cite[Theorem 3.1]{Spr}) and $\frak{M}_{2}^{I}=\ell_\infty^{I\times
I}$ isometrically.

If $M_A \in \frak{M}_{p}^{I}$ is a Schur multiplier, we have
$\bnorm{M_{\mathcal{T}_J(A)}}_{B(S_p^I)} \leq
\norm{M_{A}}_{B(S_p^I)}$ for any finite subset $J$ of $I$. Moreover,
if $M_A \in \frak{M}_{p,cb}^{I}$, we have for any finite subset $J$
of $I$ the inequality $\bnorm{M_{\mathcal{T}_J(A)}}_{CB(S_p^I)} \leq
\norm{M_{A}}_{CB(S_p^I)}$.

It is well-known that the map $(A,B) \mapsto A*B$ from $S_p^{I}
\times S_{p^*}^{I}$ into $S_1^{I}$ is contractive. In order to study
the preduals of $\frak{M}_{p}^{I}$ and $\frak{M}_{p,cb}^{I}$, we
need to show that this map is jointly completely contractive.

\begin{prop}
Suppose $1 \leq p \leq \infty$. The bilinear map
$$
\begin{array}{cccc}
    &   S_p^{I} \times S_{p^*}^{I}  &  \longrightarrow   & S_1^{I}   \\
    &   (A,B)  &  \longmapsto       &  A*B  \\
\end{array}
$$
is jointly completely contractive.
\end{prop}

\begin{proof}
We denote $\beta\colon \ell_2^I \to \ell_{\infty}^I$ the canonical
contractive map. We have
$$
\norm{\beta}_{cb,C_I \rightarrow {\ell_{\infty}^I}} =
\norm{\beta}_{\ell_{2}^I \rightarrow \ell_{\infty}^I}\leq 1 \ \ \
\text{and} \ \ \ \norm{\beta}_{cb,R_I \rightarrow \ell_{\infty}^I} =
\norm{\beta}_{\ell_{2}^I \rightarrow \ell_{\infty}^I} \leq 1
$$
(see \cite[(1.10)]{BLM}). Then by tensoring, the map $ C_I \ot_h R_I
\rightarrow \ell_{\infty}^I \ot_h \ell_{\infty}^I$ is completely
contractive. Now recall that we have a completely isometric
canonical map $\ell_{\infty}^I \ot_h \ell_{\infty}^I \rightarrow
\frak{M}_{\infty}^{I}$ and a completely isometric map $T\mapsto T^*$
from $CB(S_\infty^{I})$ into $CB(S_1^{I})$. Then the map
$$
\begin{array}{ccccccc}
 S_\infty^{I}=C_I \ot_h R_I &  \longrightarrow  &   \ell_{\infty}^I \ot_h \ell_{\infty}^I  &  \longrightarrow   &  \frak{M}_{\infty}^{I} &\longrightarrow & CB(S_1^{I}) \\
           e_{ij}  &       \longmapsto  &    e_i   \ot e_j                &  \longmapsto       &     M_{e_{ij}}                 & \longmapsto & M_{e_{ij}}\\
\end{array}
$$
is completely contractive. This means that the map $A\mapsto M_A$
from $S_\infty^{I}$ into $CB(S_1^{I})$ is completely contractive.
Then the map $(A,B)\mapsto A*B$ from $S_\infty^{I} \times S_{1}^{I}$
into $S_1^{I}$ is completely jointly contractive. By the
commutativity of $*$ and $\otp$, the map from $S_1^{I} \times
S_{\infty}^{I}$ into $S_1^{I}$ is also completely jointly
contractive. Finally, we obtain the result by bilinear interpolation
(see \cite{Pis4} page 57 and \cite{BeL} page 96).
\end{proof}

Then we can define the completely contractive map
$$
\begin{array}{cccc}
  \psi_p^{I}:  &   S_p^{I} \otp S_{p^*}^{I}  &  \longrightarrow   &  S_1^{I}  \\
               &   A \ot B  &  \longmapsto       & A*B.   \\
\end{array}
$$
As $S_p^{I} \otpb S_{p^*}^{I}$ embeds contractively into $S_p^{I}
\otp S_{p^*}^{I}$, the map $\psi_p^{I}$ induces a contraction from
$S_p^{I} \otpb S_{p^*}^{I}$ into $S_1^{I}$, which we denote by
$\varphi_{p}^{I}$. We let $\psi_p=\psi_p^{\mathbb{N}}$. The
following theorem (and the comments which follow) is a
noncommutative version of a theorem of Fig\`a-Talamanca \cite{Fig}.
This latter theorem states that the natural predual of the space of
bounded Fourier multipliers admits a concrete realization as a space
$A_p(G)$ of continuous functions on $G$. In the sequel, we consider
the dual pairs $CB(S_p^{I})$, $S_p^{I} \otp S_{p^*}^{I}$ and
$B(S_p^{I})$, $S_p^{I} \otpb S_{p^*}^{I}$ where $1 \leq p < \infty$.
\pagebreak[2]
\begin{thm}
Suppose $1 \leq p < \infty$.
\begin{enumerate}
  \item The pre-annihilator $\big(\frak{M}_{p,cb}^{I}\big)_{\perp}$ of the space $\frak{M}_{p,cb}^{I}$ of completely bounded Schur
multipliers on $S_p^{I}$ is equal to $\Ker \psi_p^{I}$. We have a
complete isometry $\frak{M}_{p,cb}^{I}=\big(S_p^{I} \otp
S_{p^*}^{I}/\Ker \psi_p^{I}\big)^* $.
  \item The pre-annihilator $\big(\frak{M}_{p}^{I}\big)_{\perp}$ of the space $\frak{M}_{p}^{I}$ of bounded Schur
multipliers on $S_p^{I}$ is equal to $\Ker \varphi_p^{I}$. We have
an isometry $\frak{M}_{p}^{I}=\big(S_{p}^{I} \otpb S_{p^*}^{I}/\Ker
\varphi_{p}^{I}\big)^* $.
\end{enumerate}
\end{thm}

\begin{proof}
We will only prove the part 1. The proof of part 2 is similar. Let
$C=\sum_{k=1}^{l} A_k \ot B_k \in S_p^{I} \ot S_{p^*}^{I}$. Note
that, for all integers $k$, we have $M_{A_k}\in \frak{M}_{p}^{I}$.
If $i,j$ are elements of $I$ we have
\begin{align*}
  \big\langle M_{e_{ij}},C\big\rangle_{CB(S_p^{I}),S_p^{I} \otp S_{p^*}^{I}}
   &= \Bigg\langle  M_{e_{ij}}, \sum_{k=1}^{l} A_k \ot B_k \Bigg\rangle_{CB(S_p^{I}),S_p^{I} \otp S_{p^*}^{I}} \\
   &= \sum_{k=1}^{l}\big\langle e_{ij} * A_k,B_k \big\rangle_{S_p^{I},S_{p^*}^{I}} \\
   &= \sum_{k=1}^{l}\big\langle e_{ij} , A_k*B_k\big\rangle_{S_p^{I},S_{p^*}^{I}}   \\
   &= \Bigg\langle e_{ij} , \sum_{k=1}^{l}A_k*B_k\Bigg\rangle\\
   &= \Big[\psi_p^{I}(C)\Big]_{ij}.
\end{align*}
By continuity, if $C \in S_p^{I} \otp S_{p^*}^{I}$, we have
$\big\langle M_{e_{ij}},C\big\rangle_{CB(S_p^{I}),S_p^{I} \otp
S_{p^*}^I}=\Big[\psi_p^I(C)\Big]_{ij}$. We deduce that, if $C \in
\Ker \psi_p^{I}$ and $M_D \in \frak{M}_{p,cb}^{I}$, we have for all
$J\in \mathcal{P}_f(I)$
$$
\big\langle M_{\mathcal{T}_J(D)},C\big\rangle_{CB(S_p^{I}),S_p^{I}
\otp S_{p^*}^{I}}=0.
$$
Now, it is easy to see that we have
$M_{\mathcal{T}_J(D)}\xrightarrow[J]{so} M_{D}$ in $CB(S_p^{I})$
$\big($i.e., for all $A \in S_p^I$, we have
$M_{\mathcal{T}_J(D)}(A)\xrightarrow[J]{} M_{D}(A)\big)$. Then
$M_{\mathcal{T}_J(D)}\xrightarrow[J]{wo} M_{D}$ in $CB(S_p^{I})$.
Moreover, recall that, for all $J\in \mathcal{P}_f(I)$, we have
$\bnorm{M_{\mathcal{T}_J(D)}}_{\frak{M}_{p,cb}^{I}} \leq
\norm{M_{D}}_{\frak{M}_{p,cb}^{I}}$. Thus
$M_{\mathcal{T}_J(D)}\xrightarrow[J]{w^*} M_{D}$. Consequently, if
$C\in \Ker \psi_p^{I}$ and $M_D\in \frak{M}_{p,cb}^{I}$ we have
$$
\big\langle M_{D},C\big\rangle_{CB(S_p^{I}),S_p^{I} \otp
S_{p^*}^{I}}=\lim_{J} \big\langle
M_{\mathcal{T}_J(D)},C\big\rangle_{CB(S_p^{I}),S_p^{I} \otp
S_{p^*}^{I}}=0.
$$
Thus we have $\Ker \psi_p^{I} \subset
\big(\frak{M}_{p,cb}^{I}\big)_{\perp}$. Now we will show that
$\big(\Ker\psi_p^{I}\big)^{\perp} \subset \frak{M}_{p,cb}^{I}$.
Suppose that $T \in \big(\Ker\psi_p^{I}\big)^{\perp}$. If $i,j,k,l$
are elements of $I$ such that $(i,j)\not=(k,l)$, the tensor $e_{ij}
\ot e_{kl}$ belongs to $\Ker\psi_p^{I}$. Therefore we have
\begin{align*}
  \big\langle T(e_{ij}),e_{kl}\big\rangle_{S_p^{I},S_{p^*}^{I}}  &= \big\langle T,e_{ij} \ot e_{kl}\big\rangle_{CB(S_p^{I}),S_p^{I} \otp S_{p^*}^{I}}  \\
   &= 0.
\end{align*}
Hence $T$ is a Schur multiplier. We conclude that
$\big(\Ker\psi_p^{I}\big)^{\perp}\subset \frak{M}_{p,cb}^{I}$. Since
$\Ker\psi_p^{I}$ is norm-closed in $S_p^{I} \otp S_{p^*}^{I}$ we
deduce that
$$
\big(\frak{M}_{p,cb}^{I}\big)_{\perp} \subset
\Big(\big(\Ker\psi_p^{I}\big)^{\perp}\Big)_{\perp}=\Ker\psi_p^{I}.
$$
Then the first claim of part 1 of the theorem is proved.

Now, we will show that $\frak{M}_{p,cb}^{I}$ is a maximal
commutative subset of $CB(S_p^{I})$. Let $T\colon S_p^{I}\to
S_p^{I}$ be a bounded map which commutes with all Schur multipliers
$M_{e_{ij}}\colon S_p^{I}\to S_p^{I}$ where $i,j\in I$. Then, for
all $i,j,k,l\in I$ such that $(i,j)\not=(k,l)$ we have
\begin{align*}
 \big\langle T(e_{ij}),e_{kl}\big\rangle_{S_p^{I},S_{p^*}^{I}}
   &= \big\langle TM_{e_{ij}}(e_{ij}),e_{kl}\big\rangle_{S_p^{I},S_{p^*}^{I}} \\
   &= \big\langle M_{e_{ij}}T(e_{ij}),e_{kl}\big\rangle_{S_p^{I},S_{p^*}^{I}} \\
   &= \big\langle T(e_{ij}),M_{e_{ij}}(e_{kl})\big\rangle_{S_p^{I},S_{p^*}^{I}}\\
   &= 0.
\end{align*}
Hence $T$ is a Schur multiplier. This proves the claim. Then
$\frak{M}_{p,cb}^{I}$ is weak* closed in $CB(S_p^I)$. We immediately
deduce the second claim of part 1 of the theorem.
\end{proof}

If $1\leq p < \infty$, we define the operator space
$\frak{R}_{p,cb}^{I}$ as the space $\Im \psi_p^{I}$ equipped with
the operator space structure of $S_p^{I} \otp S_{p^*}^{I}/\Ker
\psi_p^{I}$. We let $\frak{R}_{p,cb}=\frak{R}_{p,cb}^{\mathbb{N}}$.
We have completely isometrically
$\big(\frak{R}_{p,cb}^{I}\big)^*=\frak{M}_{p,cb}^{I}$. By
definition, we have a completely contractive inclusion
$\frak{R}_{p,cb}^{I} \subset S_1^{I}$. We define the Banach space
$\frak{R}_{p}^I$ as the space $\Im \varphi^{I}_p$ equipped with the
norm of $S_p^{I} \otpb S_{p^*}^{I}/\Ker \varphi_p^{I}$. We let
$\frak{R}_{p}=\frak{R}_{p}^{\mathbb{N}}$. We have isometrically
$\big(\frak{R}_p^{I}\big)^*=\frak{M}_p^{I}$.

By duality, well-known results on $\frak{M}_{p}^{I} $ and
$\frak{M}_{p,cb}^{I}$ translate immediately into results on
$\frak{R}_{p}^{I}$ and $\frak{R}_{p,cb}^{I}$. If $1\leq p < \infty$,
there is a contractive inclusion $\frak{R}_{p}^{I}\subset
\frak{R}_{p,cb}^{I}$. If $1 < p < \infty$, the Banach spaces
$\frak{R}_{p}^{I}$ and $\frak{R}_{p^*}^{I}$ are isometric and the
operator spaces $\frak{R}_{p,cb}^{I}$ and $\frak{R}_{p^*,cb}^{I}$
are completely isometric. We have a completely isometric isomorphism
\begin{equation}
\label{R1}
\begin{array}{cccc}
   &   \ell_1^I \ot_{h} \ell_1^I  &  \longrightarrow   &  \frak{R}_{1,cb}^{I}  \\
    & e_i \ot e_j    &  \longmapsto       &  e_{ij}  \\
\end{array}
\end{equation}
and isometric isomorphisms
$$
\begin{array}{cccc}
    &  \ell_1^I \ot_{h} \ell_1^I   &  \longrightarrow   & \frak{R}_{1}^{I}   \\
    &  e_i\ot e_j   &  \longmapsto       &   e_{ij} \\
\end{array}
 \ \ \ \text{and}\ \ \
\begin{array}{ccc}
       \ell_1^{I\times I}       &  \longrightarrow   &  \frak{R}_{2}^{I}=\frak{R}_{2,cb}^{I}   \\
               e_{ij}            &  \longmapsto       &         e_{ij}.         \\
\end{array}
$$
Suppose $1 \leq p\leq q \leq 2$, we have injective contractive maps
$$
\frak{M}_{1}^{I} \subset \frak{M}_{p}^{I}\subset \frak{M}_{q}^{I}
\subset \frak{M}_{2}^{I} \ \ \ \text{and} \ \ \ \frak{M}_{1,cb}^{I}
\subset \frak{M}_{p,cb}^{I}\subset \frak{M}_{q,cb}^{I} \subset
\frak{M}_{2,cb}^{I}
$$
(see \cite{Har} page 219). One more time, by duality, we deduce that
we have injective contractive inclusions
$$
\frak{R}_{2}^{I}\subset \frak{R}_{q}^{I}\subset
\frak{R}_{p}^{I}\subset \frak{R}_{1}^{I} \ \ \ \text{and} \ \ \
\frak{R}_{2,cb}^{I}\subset \frak{R}_{q,cb}^{I}\subset
\frak{R}_{p,cb}^{I}\subset \frak{R}_{1,cb}^{I}.
$$
Actually, the last inclusions are completely contractive. It is a
part of Proposition \ref{prop inclusion}.

Suppose $1 \leq p < \infty$. By a well-known property of the Banach
projective tensor product, an element $C$ in $S_1^I$ belongs to
$\frak{R}_p^{I}$ if and only if there exists two sequences
$(A_n)_{n\geq 1} \subset S_p^{I}$ and $(B_n)_{n\geq 1} \subset
S_{p^*}^{I}$ such that the series $\sum_{n=1}^{+\infty}A_n\ \ot B_n$
converges absolutely in $ S_p^{I} \widehat{\ot} S_{p^*}^{I}$ and $
C=\sum_{n=1}^{+\infty} A_n*B_n$ in $S_1^{I}$. Moreover, we have
\begin{align}
\label{expression norme}
\norm{C}_{\frak{R}_p^{I}}
    &= \inf\Bigg\{\ \sum_{n=1}^{+\infty}\norm{A_n}_{S_p^{I}} \norm{B_n}_{S_{p^*}^{I}}\ | \ C=\sum_{n=1}^{+\infty}A_n*B_n\Bigg\}
  \end{align}
where the infimum is taken over all possible ways to represent $C$
as before. We observe that we have an inclusion $\Mat^\fin_{I}
\subset \frak{R}_{p}^{I}$. It is clear that $\Mat^\fin_{I}$ is dense
in $\frak{R}_{p}^{I}$ and $\frak{R}_{p,cb}^{I}$.
\begin{remark}
The Banach spaces $\frak{M}_{p}^{I}$ and $\frak{M}_{p,cb}^{I}$
contain the space $\ell_\infty^I$. We deduce that, if $I$ is
infinite, the Banach spaces $\frak{M}_{p}^{I}$,
$\frak{M}_{p,cb}^{I}$, $\frak{R}_{p}^{I}$ and $\frak{R}_{p,cb}^{I}$
are not reflexive.
\end{remark}
Now we make precise the duality between the operator spaces
$\frak{M}_{p,cb}^{I}$ and $\frak{R}_{p,cb}^{I}$ on the one hand and
the Banach spaces $\frak{M}_{p}^{I}$ and $\frak{R}_{p}^{I}$ on the
other hand. Moreover, the next lemma specifies the density of
$\Mat_I^{\fin}$ in $\frak{R}_{p}^I$ and $\frak{R}_{p,cb}^I$.

\begin{lemma}
\label{duality Rp Mp}\label{lemma approx in Rp} Suppose $1 \leq p <
\infty$.
\begin{enumerate}
  \item If $J$ is a finite subset of $I$, the
truncation map $\mathcal{T}_J\colon \frak{R}_{p,cb}^I \to
\frak{R}_{p,cb}^I$ is completely contractive. Moreover, if $A\in
\frak{R}_{p,cb}^I$, we have in $\frak{R}_{p,cb}^I$
\begin{equation}\label{convergence truncation}
  \mathcal{T}_J(A)\xra[J]{}A.
\end{equation}
  \item For any completely bounded Schur multiplier $M_A\in \frak{M}_{p,cb}^{I}$ and any $B\in \frak{R}_{p,cb}^{I}$, we have
\begin{equation}\label{duality Rp-Mp}
\big\langle
M_{A},B\big\rangle_{\frak{M}_{p,cb}^{I},\frak{R}_{p,cb}^{I}}=\lim_{J}\sum_{i,j\in
J}a_{ij}b_{ij}.
\end{equation}
 \item If $J$ is a finite subset of $I$, the
truncation map $\mathcal{T}_J\colon \frak{R}_{p}^I \to
\frak{R}_{p}^I$ is contractive. Moreover, if $A\in \frak{R}_{p}^I$,
we have $ \mathcal{T}_J(A)\xra[J]{}A$ in $\frak{R}_{p}^I$.
\item For any bounded Schur multiplier $M_A\in \frak{M}_{p}^{I}$ and any $B\in \frak{R}_{p}^{I}$, we have
$\displaystyle \big\langle
M_{A},B\big\rangle_{\frak{M}_{p}^{I},\frak{R}_{p}^{I}}=\lim_{J}\sum_{i,j\in
J}a_{ij}b_{ij}.$
\end{enumerate}
\end{lemma}

\begin{proof}
We only prove the assertions for the operator space
$\frak{R}_{p,cb}^I$. If $i,j$ are elements of $I$ and $M_A \in
\frak{M}_{p,cb}^{I}$, we have
\begin{align*}
   \big\langle M_{A},e_{ij}\big\rangle_{\frak{M}_{p,cb}^{I},\frak{R}_{p,cb}^{I}}
   &=  \big\langle M_{A},e_{ij}*e_{ij}\big\rangle_{\frak{M}_{p,cb}^{I},\frak{R}_{p,cb}^{I}}\\
   &=  \big\langle M_{A}(e_{ij}),e_{ij}\big\rangle_{S_p^{I},S_{p^*}^{I}}\\
   &= a_{ij}.
\end{align*}
Then we deduce that, for all $M_A \in \frak{M}_{p,cb}^{I}$ and all
$B \in \Mat^{\fin}_I$, we have $ \big\langle
M_{A},B\big\rangle_{\frak{M}_{p,cb}^{I},\frak{R}_{p,cb}^{I}}=\sum_{i,j\in
I}a_{ij}b_{ij}$. Now, it is not difficult to see that, for any
finite subset $J$ of $I$, the truncation map $\mathcal{T}_J:S_p^I
\to S_p^I$ is completely contractive. Then, it follows easily that
the truncation map $\mathcal{T}_J\colon \frak{M}_{p,cb}^I \to
\frak{M}_{p,cb}^I$ is completely contractive. Hence, by duality and
by using the density of $\Mat^{\fin}_I$ in $\frak{R}_{p,cb}^I$, we
deduce that the truncation map $\mathcal{T}_J\colon
\frak{R}_{p,cb}^I \to \frak{R}_{p,cb}^I$ is completely contractive.
Furthermore, by density of $\Mat_I^{\fin}$ in $\frak{R}_{p,cb}^I$,
it is not difficult to prove the assertion (\ref{convergence
truncation}). Finally, the equality (\ref{duality Rp-Mp}) is now
immediate.
\end{proof}

Finally, we end the section by giving supplementary properties of
these operator spaces. For that, we need the following proposition
inspired by \cite[Proposition 2.4]{Neu}. If $x,y\in \R$, we denote
by $M_{x,y}\colon S_{p}^I\to S_{p}^I$ the Schur multiplier
associated with the matrix $\big[e^{ixr}e^{iys}\big]_{r,s\in I}$ of
$\Mat_{I}$ and by $\overline{M}_{x,y}\colon S_{p}^I\to S_{p}^I$ the
Schur multiplier associated with the matrix $\big[e^{-i xr}e^{-i
ys}\big]_{r,s\in I}$ of $\Mat_{I}$. It is easy to see that, for all
$x,y\in \R$, the maps $M_{x,y}\colon S_{p}^I\to S_{p}^I$ and
$\overline{M}_{x,y}\colon S_{p}^I\to S_{p}^I$ are completely
contractive. We denote by $dx$ the normalized measure on $[0,2\pi]$.

\begin{prop}
Suppose $1 \leq p \leq \infty$. The space $\frak{M}_{p,cb}^{I}$ of
completely bounded Schur multipliers on $S_p^I$ is 1-completely
complemented in the space $CB\big(S_p^I\big)$.
\end{prop}

\begin{proof}
Let $T\colon S_p^I\to S_p^I$ be a completely bounded map. For any
$A\in \Mat^{\fin}_I$ the map
$$
\begin{array}{cccc}
    &   [0,2\pi] \times [0,2\pi]  &  \longrightarrow   &  S_p^I  \\
    &  (x,y)   &  \longmapsto       &  M_{x,y}T\overline{M}_{x,y}(A)  \\
\end{array}
$$
is continuous and we have
\begin{align*}
\MoveEqLeft\Bgnorm{\int_{0}^{2\pi}\int_{0}^{2\pi}
M_{x,y}T\overline{M}_{x,y}(A)dxdy}_{S_p^I}
\leq \int_{0}^{2\pi}\int_{0}^{2\pi}\Bnorm{ M_{x,y}T\overline{M}_{x,y}(A)}_{S_p^I}dxdy \\
   &\leq \int_{0}^{2\pi}\int_{0}^{2\pi} \bnorm{M_{x,y}}_{S_p^I \xra{} S_p^I}\norm{T}_{S_p^I \xra{} S_p^I}\bnorm{\overline{M}_{x,y}}_{S_p^I \xra{} S_p^I}\norm{A}_{S_{p}^I}dxdy \\
   &\leq \norm{T}_{S_p^I \xra{} S_p^I}\norm{A}_{S_{p}^I}.
\end{align*}
By the previous computation, we deduce that there exists a unique
linear map $P(T)\colon S_p^I\to S_p^I$ such that for all $A\in
S_p^I$ we have
$$
\big(P(T)\big)(A)=\int_{0}^{2\pi}\int_{0}^{2\pi}M_{x,y}T\overline{M}_{x,y}(A)dxdy.
$$
Moreover, for all $\sum_{k=1}^{l}A_k\ot B_k \in \Mat_{\fin} \ot
S_p^I$ we have
\begin{align*}
\MoveEqLeft \Bgnorm{\Big(Id_{S_p}\ot
P(T)\Big)\Bigg(\sum_{k=1}^{l}A_k \ot B_k \Bigg)}_{S_p(S_p^I)}
=\Bgnorm{\sum_{k=1}^{l}A_k\ot\int_{0}^{2\pi}\int_{0}^{2\pi}M_{x,y}T\overline{M}_{x,y}(B_k)dxdy}_{S_p(S_p^I)}\\
&= \Bgnorm{\int_{0}^{2\pi}\int_{0}^{2\pi}\Big(Id_{S_p}\ot
M_{x,y}T\overline{M}_{x,y}\Big)\Bigg(\sum_{k=1}^{l}A_k\ot
B_k\Bigg)dxdy}_{S_p(S_p^I)}
\\
 & \leq \norm{T}_{cb,S_p^I \rightarrow
S_p^I}\Bgnorm{\sum_{k=1}^{l}A_k\ot B_k}_{S_p(S_p^I)}.
\end{align*}
Thus we see that the linear map $P(T)$ is actually completely
bounded and that we have $ \bnorm{P(T)}_{cb,S_p^I \rightarrow S_p^I}
\leq \norm{T}_{cb,S_p^I \rightarrow S_p^I}$. Now, for all $r,s,k,l
\in I$ we have
\begin{align*}
 \big\langle P(T)e_{rs},e_{kl}\big\rangle_{S_p^I,S_{p^*}^I}
   &= \int_{0}^{2\pi}\int_{0}^{2\pi}\Big\langle M_{x,y}T\overline{M}_{x,y}e_{rs},e_{kl}\Big\rangle_{S_p^I,S_{p^*}^I}dxdy \\
   &= \int_{0}^{2\pi}\int_{0}^{2\pi} e^{-i xr}e^{-i ys}\Big\langle M_{x,y}Te_{rs},e_{kl}\Big\rangle_{S_p^I,S_{p^*}^I}dxdy \\
   &= \Bigg(\int_{0}^{2\pi}\int_{0}^{2\pi} e^{-i xr}e^{-i ys}e^{i xk}e^{i yl} dxdy\Bigg) \big\langle Te_{rs},e_{kl}\big\rangle_{S_p^I,S_{p^*}^I}\\
   &= \Bigg(\int_{0}^{2\pi} e^{i x(k-r)}dx\Bigg)\Bigg(\int_{0}^{2\pi}e^{i y(l-s)}dy\Bigg)\big\langle Te_{rs},e_{kl}\big\rangle_{S_p^I,S_{p^*}^I}\\
   &= \delta_{rk}\delta_{sl}\big\langle T(e_{rs}),e_{kl}\big\rangle_{S_p^I,S_{p^*}^I}.
\end{align*}
Then the linear map $P(T)\colon S_p^{I}\to S_p^{I}$ is a Schur
multiplier. Moreover, if $T\colon S_p^{I}\to S_p^{I}$ is a Schur
multiplier, we have $P(T)=T$.

Now, if $T \in M_n\big(CB(S_p^I)\big)$ and $[A_{kl}]_{1\leq k,l\leq
m}\in M_m\big(S_p^I\big)$, with the notations of Lemma
\ref{lemmaVW},  we have
\begin{align*}
\MoveEqLeft \Biggl\| \biggl[\int_{0}^{2\pi}\int_{0}^{2\pi} M_{x,y}T_{ij}\overline{M}_{x,y}(A_{kl})dx dy\biggr]_{\raisebox{2pt}{\ensuremath{\substack{1\leq i,j\leq n\\ 1 \leq k,l\leq m}}}}\Biggr\|_{M_{mn}(S_p^I)} \\
   &\leq   \int_{0}^{2\pi} \int_{0}^{2\pi}\bgnorm{\Big[ M_{x,y}T_{ij}\overline{M}_{x,y}\Big]_{1\leq i,j\leq n}}_{M_n(CB(S_p^I))} \bnorm{[A_{kl}]}_{}dx dy \\
   &=       \int_{0}^{2\pi} \int_{0}^{2\pi}\bgnorm{\Big(Id_{M_n}\ot\Theta_{M_{x,y},\overline{M}_{x,y}}\Big)(T)}_{M_n(CB(S_p^I))}\bnorm{[A_{kl}]}dx dy \\
   &\leq   \norm{T}_{M_n(CB(S_p^I))}\bnorm{[A_{kl}]_{1\leq k,l\leq m}}_{M_m(S_p^I)}\hspace{1cm} \text{by Lemma \ref{lemmaVW}}.
\end{align*}
Thus we obtain
\begin{align*}
\bnorm{(Id_{M_n}\ot P)(T)}_{M_n(CB(S_p^I))}
   &=     \Bnorm{\big[P(T_{ij})\big]_{1 \leq i,j\leq n}}_{M_n(CB(S_p^I))} \\
   &\leq   \norm{T}_{M_n(CB(S_p^I))}.
\end{align*}
We deduce that the map $P\colon CB\big(S_p^I\big) \to
\frak{M}_{p,cb}^{I}$ is completely contractive.
 The proof is complete.
\end{proof}

\begin{prop}
\label{prop inclusion}
\begin{enumerate}
  \item We have completely isometric isomorphisms
$$
\begin{array}{cccc}
    &  \ell_1^I \otp \ell_1^I   &  \longrightarrow   & \frak{R}_{2,cb}^{I}   \\
    &  e_i\ot e_j   &  \longmapsto       &   e_{ij} \\
\end{array}
 \ \ \ \text{and}\ \ \
\begin{array}{cccc}
    &   \ell_\infty^{I \times I}  &  \longrightarrow   &  \frak{M}_{2,cb}^{I}   \\
    &             A               &  \longmapsto       &         M_A.         \\
\end{array}
$$
  \item Suppose $1 \leq p\leq q \leq 2$. We have injective completely contractive maps
$$
\frak{M}_{1,cb}^{I}\subset \frak{M}_{p,cb}^{I}\subset
\frak{M}_{q,cb}^{I}\subset \frak{M}_{2,cb}^{I} \ \ \ \text{and}\ \ \
\frak{R}_{2,cb}^{I}\subset \frak{R}_{q,cb}^{I}\subset
\frak{R}_{p,cb}^{I}\subset \frak{R}_{1,cb}^{I}.
$$
\end{enumerate}
\end{prop}

\begin{proof}
1) By minimality, we have a completely contractive map
$\frak{M}_{2,cb}^{I} \rightarrow \ell_\infty^{I \times I}$. We will
show that the inverse map is completely contractive. We have a
complete isometry
$$
\begin{array}{cccc}
    &   \ell_\infty^{I\times I}  &  \longrightarrow   &  B\big(S_2^I\big)= CB(C_{I\times I})  \\
    &               A                      &  \longmapsto       &           M_A.   \\
\end{array}
$$
Now we know that $(R_{I\times I})^*=C_{I\times I}$. Then we deduce a
complete isometry
$$
\begin{array}{ccccc}
   \ell_\infty^{I\times I}  &  \longrightarrow   &        CB(C_{I\times I})        &  \longrightarrow  &  CB(R_{I\times I})  \\
               A                      &  \longmapsto       &      M_A      &  \longmapsto      &  (M_A)^*=M_A. \\
\end{array}
$$
By interpolation, we deduce a complete contraction
$$
\ell_\infty^{I\times I} \rightarrow \big(CB(C_{I\times
I}),CB(R_{I\times I}) \big)_{\frac{1}{2}}.
$$  Recall that we have
$\big(C_{I\times I},R_{I\times I}\big)_{\frac{1}{2}}=S_2^{I}$
completely isometrically (see \cite{Pis2} pages 137 and 140). Then
we have a complete contraction
$$
\big(CB(C_{I\times I}),CB(R_{I\times I}) \big)_{\frac{1}{2}}
\rightarrow CB\big(S_2^{I}\big).
$$
Finally, we obtain a complete contraction $\ell_\infty^{I\times
I}\xra{}CB\big(S_2^{I}\big)$. We obtain the other isomorphism by
duality.

2) Let $1 \leq p\leq q \leq 2$. Recall that we have a contraction
from $\frak{M}_{p,cb}^{I}$ into $\frak{M}_{2,cb}^{I}$ (see
\cite{Har} page 219). Moreover we have
$\frak{M}_{2,cb}^{I}=\ell_\infty^{I\times I}$ completely
isometrically. Thus we have a complete contraction
$\frak{M}_{p,cb}^{I} \rightarrow \frak{M}_{2,cb}^{I}$. Now, there
exists $0 \leq \theta \leq 1$ with
$S_q^{I}=\big(S_p^{I},S_2^{I}\big)_\theta$. Moreover, the identity
mapping $ \frak{M}_{p,cb}^{I}\rightarrow\frak{M}_{p,cb}^{I}$ is
completely contractive. By interpolation, we obtain a complete
contraction $ \frak{M}_{p,cb}^{I} \rightarrow
\big(\frak{M}_{p,cb}^{I},\frak{M}_{2,cb}^{I}\big)_{\theta}$. On one
hand, we know that we have a complete contraction
$$
\Big(CB\big(S_p^{I}\big),CB\big(S_2^{I}\big)\Big)_\theta \rightarrow
CB\Big(\big(S_p^{I},S_2^{I}\big)_\theta\Big)=CB\big(S_q^{I}\big).
$$
On the other hand, the space $\frak{M}_{p,cb}^{I}$ of completely
bounded Schur multipliers is 1-completely complemented in the space
$CB\big(S_p^I\big)$. Then we have a complete contraction
$\big(\frak{M}_{p,cb}^{I},\frak{M}_{2,cb}^{I}\big)_{\theta}
\rightarrow \frak{M}_{q,cb}^{I}$. By composition, we deduce that we
have a complete contraction $\frak{M}_{p,cb}^{I} \subset
\frak{M}_{q,cb}^{I}$. We obtain the other completely contractive
maps by duality.
\end{proof}


\section{Non commutative Fig\`a-Talamanca-Herz algebras}
We begin with the cases $p=1$ and $p=2$. Recall that we have a
completely isometric isomorphism $\frak{R}_{1,cb}^I=\ell_1^I \ot_{h}
\ell_1^I$ (see (\ref{R1})) and a completely contractive inclusion
$\frak{R}_{1,cb}^I\subset S_1^I$. Hence, the trace on $S_1^I$
induces a completely contractive functional
$$
\begin{array}{cccc}
  \tr:  &  \ell_1^I \ot_{h} \ell_1^I   &  \longrightarrow   &  \mathbb{C}  \\
        &  e_i \ot e_j   &  \longmapsto       &  \delta_{ij}.  \\
\end{array}
$$
By tensoring, we deduce a completely contractive map
$$
Id_{\ell_1^I} \ot \tr \ot Id_{\ell_1^I}\colon \ell_1^I \ot_{h}
\ell_1^I \ot_{h} \ell_1^I \ot_{h} \ell_1^I \to \ell_1^I \ot_{h}
\ell_1^I.
$$
By composition with the canonical completely contractive map
$$
\big(\ell_1^I \ot_{h} \ell_1^I\big) \otp \big( \ell_1^I \ot_{h}
\ell_1^I\big)\xra{} \ell_1^I \ot_{h} \ell_1^I \ot_{h} \ell_1^I
\ot_{h} \ell_1^I
$$
we obtain a completely contractive map
$$
Id_{\ell_1^I} \ot \tr \ot Id_{\ell_1^I}\colon \big(\ell_1^I \ot_{h}
\ell_1^I\big) \otp \big( \ell_1^I \ot_{h} \ell_1^I\big) \to \ell_1^I
\ot_{h} \ell_1^I .
$$
With the identification $\frak{R}_{1,cb}^I=\ell_1^I \ot_{h}
\ell_1^I$, we obtain the completely contractive map
$$
\begin{array}{cccc}
    & \frak{R}_{1,cb}^I \otp \frak{R}_{1,cb}^I  &  \longrightarrow   & \frak{R}_{1,cb}^I  \\
    &  A \ot B   &  \longmapsto       &  AB.  \\
\end{array}
$$
This means that the space $\frak{R}_{1,cb}^I$ equipped with the
matricial product is a completely contractive Banach algebra. Now,
recall that we have $\frak{R}_{2,cb}^I=\ell_1^I \otp \ell_1^I$
completely isometrically. Then, by a similar argument,
$\frak{R}_{2,cb}^I$ equipped with the matricial product is also a
completely contractive Banach algebra. For other values of $p$, the
proof is more complicated since we do not have any explicit
description of $\frak{R}_{p,cb}^I$.

In the following proposition, we give a link between
$\frak{R}_{p,cb}^I$ and $\frak{R}_{p,cb}^{I\times I}$.

\begin{prop}
\label{link RpI and RpI2} Suppose $1\leq p < \infty$. Then there
exists a canonical complete contraction
$$
\begin{array}{cccc}
    &  \frak{R}_{p,cb}^I \otp \frak{R}_{p,cb}^I   &  \longrightarrow   &  \frak{R}_{p,cb}^{I\times I}  \\
    &   A \ot B  &  \longmapsto       & A \ot B.   \\
\end{array}
$$
\end{prop}

\begin{proof}
The identity mapping on $S_p^I \ot S_p^I$ extends to a complete
contraction $S_{p}^I\otp S_{p}^I \to S_{p}^I(S_{p}^I)$. Hence by
tensoring, we obtain a completely contractive map
$$
\beta\colon S_{p}^I \otp S_{p}^I \otp S_{p^*}^I \otp S_{p^*}^I \to
S_{p}^I(S_{p}^I) \otp S_{p^*}^I(S_{p^*}^I).
$$
The map $\psi_p^I\colon S_{p}^I \otp
S_{p^*}^I\xra{}\frak{R}_{p,cb}^I$ is a complete quotient map. By
\cite[Proposition 7.1.7]{ER1}, we obtain a complete quotient map
$$
\psi_p^I \ot \psi_p^I\colon S_{p}^I \otp S_{p^*}^I \otp S_{p}^I
\otp S_{p^*}^I\to \frak{R}_{p,cb}^I \otp \frak{R}_{p,cb}^I.
$$
Finally, by the commutativity of $\otp$, the map
$$
\begin{array}{cccc}
  \alpha:  &  S_{p}^I \otp S_{p^*}^I \otp S_{p}^I \otp S_{p^*}^I   &  \longrightarrow   &  S_{p}^I \otp S_{p}^I \otp S_{p^*}^I \otp S_{p^*}^I  \\
    &  A\ot B\ot C \ot D   &  \longmapsto       &  A \ot C \ot B \ot D  \\
\end{array}
$$
is completely isometric. We will prove that there exists a unique
linear map such that the following diagram is commutative and that
this map is completely contractive.
$$
\xymatrix @R=2cm @C=1cm{
    S_{p}^I \otp S_{p^*}^I \otp S_{p}^I \otp S_{p^*}^I \ar[r]^{\alpha} \ar[d]_{\psi_p^I \ot \psi_p^I} & S_{p}^I \otp S_{p}^I \otp S_{p^*}^I \otp S_{p^*}^I \ar[r]^{\beta} & S_{p}^I(S_{p}^I) \otp S_{p^*}^I(S_{p^*}^I) \ar[d]^{\psi_{p}^{I\times I}}\\
   \frak{R}_{p,cb}^I \otp \frak{R}_{p,cb}^I \ar@{.>}[rr]_{}                                               &                             & \frak{R}_{p,cb}^{I\times I}
  }
$$
We have $\frak{R}_{p,cb}^I \otp \frak{R}_{p,cb}^I =\Big(S_{p}^I \otp
S_{p^*}^I \otp S_{p}^I \otp S_{p^*}^I\Big)/\Ker\big(\psi_p^I \ot
\psi_p^I\big)$ completely isometrically. It suffices to show that
$\Ker \big(\psi_p^I \ot \psi_p^I\big) \subset
\Ker\big(\psi_{p}^{I\times I}\beta \alpha\big) $. By
\cite[Proposition 7.1.7]{ER1} , we have the equality
$$
\Ker \big(\psi_p^I \ot \psi_p^I\big)=\clos
\Big(\Ker\big(\psi_p^I\big) \ot S_{p}^I \otp S_{p^*}^I+ S_{p}^I \otp
S_{p^*}^I\ot \Ker\big(\psi_p^I\big) \Big).
$$
Since the space $\Ker\big(\psi_{p}^{I\times I}\beta \alpha\big)$ is
closed in $S_{p}^I \otp S_{p^*}^I \otp S_{p}^I \otp S_{p^*}^I$, it
suffices to show that
$$
\Ker\big(\psi_p^I\big) \ot S_{p}^I \otp S_{p^*}^I+ S_{p}^I \otp
S_{p^*}\ot \Ker\big(\psi_p^I\big)\subset \Ker\big(\psi_{p}^{I\times
I}\beta \alpha\big).
$$
Let $E \in \Ker\big(\psi_p^I\big) \ot  S_{p}^I \otp S_{p^*}^I$.
There exists integers $n_i, m_j$, matrices $A_{k,i},C_{l,j}\in
S_{p}^I$ and $B_{k,i},D_{l,j} \in S_{p^*}^I$ such that the sequences
$$
\Bigg(\sum_{k=1}^{n_i} A_{k,i}\ot B_{k,i} \Bigg)_{i \geq 1}  \ \ \ \
\ \text{and} \ \ \ \ \ \Bigg(\sum_{l=1}^{m_j} C_{l,j}\ot
D_{l,j}\Bigg)_{j\geq 1}
$$
are convergent in $S_p^I \otp S_{p^*}^I$,
$$
E=\Bigg(\lim_{i\to +\infty}\sum_{k=1}^{n_i} A_{k,i}\ot B_{k,i}
\Bigg)\ot\Bigg( \lim_{j\to +\infty}\sum_{l=1}^{m_j} C_{l,j}\ot
D_{l,j}\Bigg)
$$
and
$$
\psi_{p}^I\Bigg(\lim_{i \to
+\infty}\sum_{k=1}^{n_i} A_{k,i}\ot B_{k,i}\Bigg)= 0.
$$
Then, in the space $S_1^I$, we have
\begin{align}
 \sum_{k=1}^{n_i} A_{k,i}* B_{k,i}   &\xra[i \to +\infty]{}0.
\end{align}
Moreover, note that, by continuity of the map $\psi_p^I\colon S_p^I
\otp S_{p^*}^I \to S_1^I$, the sequence $\displaystyle
\Bigg(\sum_{l=1}^{m_j} C_{l,j}* D_{l,j}\Bigg)_{j\geq 1}$ is
convergent. Now, we have
\begin{align*}
\MoveEqLeft \psi_{p}^{I\times I}\beta \alpha(E)
    = \psi_{p}^{I\times I}\beta \alpha\Bigg(\Bigg(\lim_{i \to +\infty}\sum_{k=1}^{n_i} A_{k,i}\ot B_{k,i}\Bigg)\ot\Bigg( \lim_{j \to +\infty}\sum_{l=1}^{m_j} C_{l,j}\ot D_{l,j}\Bigg)\Bigg) \\
   &= \lim_{i \to +\infty}\lim_{j \to +\infty} \sum_{k=1}^{n_i}\sum_{l=1}^{m_j}\psi_{p}^{I\times I}\beta \alpha\big( A_{k,i}\ot B_{k,i} \ot C_{l,j}\ot D_{l,j}\big) \\
   &= \lim_{i \to +\infty}\lim_{j \to +\infty} \sum_{k=1}^{n_i}\sum_{l=1}^{m_j}\psi_{p}^{I\times I}\big( A_{k,i}\ot C_{l,j} \ot B_{k,i}\ot D_{l,j}\big)\\
   &= \lim_{i \to +\infty}\lim_{j \to +\infty} \sum_{k=1}^{n_i}\sum_{l=1}^{m_j}  \big(A_{k,i}\ot C_{l,j}\big)*\big(B_{k,i}\ot D_{l,j}\big)\\
   &= \lim_{i \to +\infty}\lim_{j \to +\infty} \sum_{k=1}^{n_i}\sum_{l=1}^{m_j}  \big(A_{k,i}* B_{k,i}\big)\ot \big(C_{l,j}* D_{l,j}\big)\\
   &= \Bigg(\lim_{i \to +\infty} \sum_{k=1}^{n_i} A_{k,i}* B_{k,i}\Bigg) \ot\Bigg(\lim_{j \to +\infty}\sum_{l=1}^{m_j}  C_{l,j}* D_{l,j}\Bigg)\\
   &= 0\hspace{1cm} \text{by (3.1)}.
\end{align*}
We prove that $S_{p}^I \otp S_{p^*}^I\ot \Ker\big(\psi_p^I\big)
\subset \Ker\big(\psi_{p}^{I\times I}\beta \alpha\big)$ by a similar
computation. The proof is complete.
\end{proof}

Now, we define the map $V\colon \Mat_I^\fin \ot \Mat_I^\fin \xra{}
\Mat_I^\fin \ot \Mat_I^\fin$ by $V(e_{ij}\ot e_{kl})= \delta_{kl}\
e_{ik}\ot e_{kj}$.

\begin{prop}
With respect to trace duality, the map $W\colon \Mat_I^\fin \ot
\Mat_I^\fin \to \Mat_I^\fin \ot \Mat_I^\fin$ defined by
$$
W(e_{ij}\ot e_{kl})= \delta_{jk}\ e_{il}\ot e_{jj}
$$
is the dual map of $V$. Moreover, the map $V$ induces a partial
isometry $V\colon S_{2}^I \ot_{2} S_{2}^I \to S_{2}^I \ot_{2}
S_{2}^I$.
\end{prop}

\begin{proof}
For all $i,j,k,l,r,s,t,u\in I$, we have
\begin{align*}
 \tr \Big(V(e_{ij}\ot e_{kl})(e_{rs}\ot e_{tu})^T\Big)
   &= \delta_{kl} \tr\Big((e_{ik}\ot e_{kj})\big(e_{rs}^T\ot e_{tu}^T\big)\Big)\\
   &= \delta_{kl} \tr\big(e_{ik}e_{rs}^T\big) \tr\big(e_{kj}e_{tu}^T\big) \\
   &=  \delta_{klst}\delta_{ir}\delta_{ju}
\end{align*}
and
\begin{align*}
 \tr\Big((e_{ij}\ot e_{kl})\big(W(e_{rs}\ot e_{tu})\big)^T\Big)
   &= \delta_{st} \tr\Big((e_{ij}\ot e_{kl})(e_{ru}\ot e_{ss})^T\Big)\\
   &= \delta_{st} \tr\big(e_{ij}e_{ru}^T\big) \tr\big(e_{kl}e_{ss}^T\big) \\
   &=  \delta_{klst}\delta_{ir}\delta_{ju}.
\end{align*}
We conclude that $W$ is the dual map of $V$.  The fact that $V$
induces a partial isometry is clear.
\end{proof}

\begin{prop}
\label{V and W extend to Sp} Suppose $1 \leq p \leq \infty$. The
maps $V\colon \Mat_I^\fin \ot \Mat_I^\fin \to \Mat_I^\fin \ot
\Mat_I^\fin$ and $W\colon \Mat_I^\fin \ot \Mat_I^\fin \to
\Mat_I^\fin \ot \Mat_I^\fin$ admit completely contractive extensions
$V\colon S_{p}^I(S_{p}^I) \to S_{p}^I(S_{p}^I)$ and $W\colon
S_{p}^I(S_{p}^I) \to S_{p}^I(S_{p}^I)$.
\end{prop}

\begin{proof}
We first prove that $V$ and $W$ admit completely contractive
extensions from $S_\infty^I(S_\infty^I)$ into
$S_\infty^I(S_\infty^I)$. Suppose that $B= \sum_{i,j,k,l\in J}
b_{ijkl} \ot e_{ij} \ot e_{kl} \in \Mat_\fin \ot \Mat_I^\fin
\ot\Mat_I^\fin$ with $J\in \mathcal{P}_f(I)$ and $b_{ijkl} \in
\Mat_\fin$ for all $i,j,k,l\in J$. Note that the matrix
$U=\displaystyle\sum_{r,s\in J} e_{rs}\ot e_{sr}\displaystyle$ of
$S_{\infty}^J(S_{\infty}^J)$ is unitary. Then we have
\begin{align*}
 \MoveEqLeft  \Bnorm{ (Id_{S_\infty} \ot V)(B)}_{S_\infty(S_\infty^I(S_\infty^I))} =\Bgnorm{\sum_{i,j,k\in J} b_{ijkk}\ot e_{ik} \ot e_{kj}}_{S_\infty(S_\infty^I(S_\infty^I))}\\
   &= \Bgnorm{\Bigg(I_{S_\infty}\ot\Bigg(\sum_{r,s\in J} e_{rs}\ot e_{sr}\Bigg)\Bigg)\Bigg(\sum_{i,j,k\in J} b_{ijkk}\ot  e_{ik} \ot e_{kj}\Bigg)}_{S_\infty(S_\infty^I(S_\infty^I))} \\
   &= \Bgnorm{\sum_{r,s,i,j,k\in J} b_{ijkk} \ot e_{rs}e_{ik} \ot e_{sr}e_{kj}}_{S_\infty(S_\infty^I(S_\infty^I))} \\
   &= \Bgnorm{\sum_{i,j,k\in J} b_{ijkk}\ot  e_{kk} \ot e_{ij}}_{S_\infty(S_\infty^I(S_\infty^I))} \\
   &= \Bgnorm{\sum_{k\in J} e_{kk} \ot \Bigg(\sum_{i,j\in I} b_{ijkk}\ot  e_{ij}\Bigg)}_{S_\infty^I(S_\infty(S_\infty^I))} \\
   &= \max_{ k \in J}  \Bgnorm{\sum_{i,j\in I} b_{ijkk}\ot e_{ij}}_{S_\infty(S_\infty^I)} \\
   & \leq  \norm{B}_{S_\infty(S_\infty^I(S_\infty^I))}\hspace{0.4cm} \text{(submatrices)}
\end{align*}
and
\begin{align*}
\MoveEqLeft \Bnorm{ (Id_{S_\infty} \ot
W)(B)}_{S_\infty(S_\infty^I(S_\infty^I))}
   =  \Bgnorm{\sum_{i,j,l\in J} b_{ijjl}\ot e_{il} \ot e_{jj}}_{S_\infty(S_\infty^I(S_\infty^I))}\\
   &= \Bgnorm{\big(I_{S_\infty} \ot U\big)\Bigg(\sum_{i,j,l\in J} b_{ijjl}\ot  e_{il} \ot e_{jj}\Bigg)\big(I_{S_\infty}\ot U\big)}_{S_\infty(S_\infty^I(S_\infty^I))} \\ 
   &= \Bgnorm{\sum_{r,s,i,j,l,t,u\in J} b_{ijjl} \ot e_{rs}e_{il}e_{tu} \ot e_{sr}e_{jj}e_{ut}}_{S_\infty(S_\infty^I(S_\infty^I))} \\
   &= \Bgnorm{\sum_{i,j,l\in J} b_{ijjl} \ot  e_{jj} \ot e_{il}}_{S_\infty(S_\infty^I(S_\infty^I))} \\
   &= \Bgnorm{\sum_{j \in J}e_{jj} \ot \Bigg(\sum_{i,l\in J} b_{ijjl}\ot  e_{il}\Bigg)}_{S_\infty^I(S_\infty(S_\infty^I))} \\
   &= \max_{j \in J}  \Bgnorm{\sum_{i,l\in J} b_{ijjl}\ot e_{il}}_{S_\infty(S_\infty^I)}\\
   &\leq \Bgnorm{\sum_{i,j,k,l\in J} b_{ijkl} \ot e_{kj} \ot e_{il}}_{S_\infty(S_\infty^I(S_\infty^I))}\hspace{0.4cm} \text{(submatrices)}\\
   &= \Bgnorm{\Bigg(I_{S_\infty}\ot\Bigg(\sum_{r,s\in J} e_{rs}\ot e_{sr}\Bigg)\Bigg)\Bigg(\sum_{i,j,k,l\in J} b_{ijkl} \ot e_{kj} \ot e_{il}\Bigg)}_{S_\infty(S_\infty^I(S_\infty^I))}  \\
   &= \Bgnorm{\sum_{r,s,i,j,k\in J} b_{ijkl} \ot e_{rs}e_{kj} \ot e_{sr}e_{il}}_{S_\infty(S_\infty^I(S_\infty^I))} \\
   &= \norm{B}_{S_\infty(S_\infty^I(S_\infty^I))}\\
\end{align*}
Then we deduce the claim. Hence, by duality, the maps $V^*\colon
S_{1}^I\big(S_{1}^I\big) \to S_{1}^I\big(S_{1}^I\big)$ and
$W^*\colon S_{1}^I\big(S_{1}^I\big) \to S_{1}^I\big(S_{1}^I\big)$
are completely contractive. Moreover, we know that $W=V^*$. By
interpolation between $p=1$ and $p=\infty$, we obtain that the maps
$V\colon S_{p}^I\big(S_{p}^I\big) \to S_{p}^I\big(S_{p}^I\big)$ and
$W\colon S_{p}^I\big(S_{p}^I\big) \to S_{p}^I\big(S_{p}^I\big)$ are
completely contractive.
\end{proof}

Now, we define the linear map
$$
\begin{array}{cccc}
  \Delta:  &  \Mat_I   &  \longrightarrow   &  \Mat_{I\times I}  \\
           &   A  &  \longmapsto            &  [a_{ts}\delta_{ur}]_{(t,r),(u,s) \in I\times I}.  \\
\end{array}
$$

\begin{prop}
\label{the diag commutes} Let $1 \leq p \leq \infty$. Suppose that
$M_A\colon S_p^I \to S_p^I$ is a completely bounded  Schur
multiplier on $S_p^I$ associated with a matrix $A$ of $\Mat_I$. Then
the map $V\big(M_{A} \ot Id_{S_p^I}\big)W$ is a bounded Schur
multiplier on $S_p^I(S_p^I)$. Its associated matrix is $\Delta(A)$.
\end{prop}

\begin{proof}
If $i,j,k,l \in I$ and $M_A \in \frak{M}_{p,cb}^I$, we have
\begin{align*}
  M_{\Delta(A)}(e_{ij}\ot e_{kl})
   &= \Big([a_{ts}\delta_{ur}]_{(t,r),(u,s) \in I\times I}\Big)* \Big([\delta_{it}\delta_{ju}\delta_{kr}\delta_{ls}]_{(t,r),(u,s) \in I\times I}\Big)\\
   &= \delta_{jk} a_{il}\Big([\delta_{it}\delta_{ju}\delta_{kr}\delta_{ls}]_{(t,r),(u,s) \in I\times I}\Big)\\
   &= \delta_{jk} a_{il}e_{ik}\ot e_{kl}
\end{align*}
and
\begin{align*}
  V\big(M_{A} \ot Id_{S_p^I}\big)W(e_{ij}\ot e_{kl})
   &=  \delta_{jk}V\big(M_{A} \ot Id_{S_p^I}\big)(e_{il}\ot e_{jj})\\
   &=  \delta_{jk}a_{il}V(e_{il}\ot e_{kk})\\
   &=  \delta_{jk}a_{il} e_{ik} \ot e_{kl}.\qedhere
\end{align*}
\end{proof}

Recall that, for all operator spaces $E$ and $F$, the map $R\ot T
\mapsto R \ot T$ is completely contractive  from $CB(E) \otp CB(F)$
into $CB\big(E \ot_{\min} F\big)$ and from $CB(E) \otp CB(F)$ into
$CB\big(E \otp F\big)$ (see \cite[Proposition 5.11]{BlP}).

\begin{prop}
\label{theo tensor By IdSp}
Suppose $1 \leq p \leq \infty$. Let
$I,J$ be any sets. The map
$$
\begin{array}{cccc}
    &   CB\big(S_p^I\big)  &  \longrightarrow   &  CB\big(S_p^I(S_p^J)\big)  \\
    &   T  &  \longmapsto       & T \ot Id_{S_p^J}   \\
\end{array}
$$
is a complete contraction.
\end{prop}

\begin{proof}
By definition, we have
$S_\infty^J\big(S_p^I\big)=S_\infty^J\ot_{\min}S_p^I$ and
$S_1^J\big(S_p^I\big)=S_1^J\otp S_p^I$ completely isometrically.
Then we obtain two complete contractions
$$
\begin{array}{ccccc}
  CB\big(S_p^I\big)  &  \longrightarrow   & CB\big(S_\infty^J\big) \otp CB\big(S_p^I\big)     & \longrightarrow  &   CB\big(S_\infty^J(S_p^I)\big)    \\
    T                &  \longmapsto       &    Id_{S_\infty^J} \ot   T      & \longmapsto      &     Id_{S_\infty^J}\ot T  \\
\end{array}
$$
and
$$
\begin{array}{ccccc}
  CB\big(S_p^I\big)  &  \longrightarrow   &  CB\big(S_1^J\big) \otp CB\big(S_p^I\big)    & \longrightarrow  &   CB\big(S_1^J(S_p^I)\big)    \\
    T                &  \longmapsto       &      Id_{S_1^J} \ot T       & \longmapsto      &   Id_{S_1^J}\ot T  . \\
\end{array}
$$
By interpolation, we obtain a completely contractive map
$$
CB\big(S_p^I\big)\xra{}\Big(CB\big(S_\infty^J(S_p^I)\big),CB\big(S_1^J(S_p^I)\big)\Big)_{\frac{1}{p}}.
$$
We conclude by composing with the complete contraction
$$
\Big(CB\big(S_\infty^J(S_p^I)\big),CB\big(S_1^J(S_p^I)\big)\Big)_{\frac{1}{p}}
\xra{} CB\big(S_p^J(S_p^I)\big)
$$
and by using the Fubini's theorem (see \cite[Theorem 1.9]{Pis2}).
\end{proof}

\begin{remark}
It is easy to see that this map is completely isometric.
\end{remark}

The next theorem is the principal result of this paper.
\pagebreak[2]
\begin{thm}
\label{th principal} Suppose $1 \leq p <\infty$. The space
$\frak{R}_{p,cb}^I$ equipped with the usual matricial product is a
completely contractive Banach algebra. More precisely, if $A$ and
$B$ are matrices of $\frak{R}_{p,cb}^I$ and $i,j\in I$, the limit
$\lim_{J}\sum_{k\in J} a_{ik}b_{kj}$ exists. Moreover, the matrix
$A.B$ of $\Mat_I$ defined by $[A.B]_{ij}=\lim_{J}\sum_{k\in J}
a_{ik}b_{kj}$ belongs to $\frak{R}_{p,cb}^I$. Finally, the map
$$
\begin{array}{cccc}
    &  \frak{R}_{p,cb}^I \otp \frak{R}_{p,cb}^I   &  \longrightarrow   &   \frak{R}_{p,cb}^I   \\
    &                 A\ot B                  &  \longmapsto       &       AB           \\
\end{array}
$$
is completely contractive.
\end{thm}

\begin{proof}
We have already seen that it suffices to prove the theorem with $1<
p<\infty$. If $M_A \in \frak{M}_{p,cb}^I$, by Proposition \ref{the
diag commutes}, we have the following commutative diagram
$$
\xymatrix @R=2cm @C=3cm{
    S_{p}^I(S_{p}^I) \ar[r]^{M_{\Delta(A)}} \ar[d]_{W}  & S_{p}^I(S_{p}^I)\\
    S_{p}^I(S_{p}^I) \ar[r]_{M_{A}\ot Id_{S_{p}^I}} & S_{p}^I(S_{p}^I).\ar[u]_{V}\\
  }
$$
By Proposition \ref{theo tensor By IdSp}, the map $M_A\mapsto M_A
\ot Id_{S_p^I}$ is completely contractive from $\frak{M}_{p,cb}^I$
into $\frak{M}_{p,cb}^{I\times I}$. Moreover it is easy to see that
this map is w*-continuous. Since $S_p^I\big(S_p^I\big)$ is
reflexive, by Lemma \ref{lemmaVW} and by composition, the map $M_A
\mapsto M_{\Delta(A)}$ from $\frak{M}_{p,cb}^I$ into
$\frak{M}_{p,cb}^{I\times I}$ is a complete contraction and is
w*-continuous. We denote by $\Delta_{*}\colon
\frak{R}_{p,cb}^{I\times I} \xra{} \frak{R}_{p,cb}^I$ its
preadjoint. Now, by Lemma \ref{lemma approx in Rp}, we have for all
$i,j\in I$ and for all matrices $A,B$ of $\Mat_I^{\fin}$
\begin{align*}
   \big[\Delta_{*}(A\ot B)\big]_{ij}
   &=  \Big\langle M_{e_{ij}},\Delta_{*}(A \ot B)\Big\rangle_{\frak{M}_{p,cb}^I,\frak{R}_{p,cb}^I} \\
   &=  \Big\langle M_{\Delta(e_{ij})},A \ot B\Big\rangle_{\frak{M}_{p,cb}^{I\times I},\frak{R}_{p,cb}^{I\times I}} \\
   &=  \Big\langle M_{[\delta_{it}\delta_{js}\delta_{ur}]_{(t,r),(u,s) \in I\times I}}, [a_{tu}b_{rs}]_{(t,r),(u,s) \in I\times I}\Big\rangle_{\frak{M}_{p,cb}^{I\times I},\frak{R}_{p,cb}^{I\times I}}\\
   &=  \lim_{J}\sum_{r \in J} a_{ir}b_{rj}\\
   &=  [A.B]_{ij}.
\end{align*}
Thus we conclude that, if $A,B \in \Mat_I^{\fin}$, we have
$\Delta_{*}(A\ot B)=AB$. By Proposition \ref{link RpI and RpI2} and
by density of $\Mat_I^{\fin} \ot \Mat_I^{\fin}$ in
$\frak{R}_{p,cb}^I \otp \frak{R}_{p,cb}^I$, we deduce that the map
$$
\begin{array}{cccccc}
 \Mat_I^{\fin} \ot \Mat_I^{\fin}     &     \longrightarrow     &  \frak{R}_{p,cb}^{I\times I}   &  \xra{\Delta_{*}}& \frak{R}_{p,cb}^I \\
  A \ot B                        &       \longmapsto       &         A \ot B                   &     \longmapsto  &       AB       \\
\end{array}
$$
admits a unique bounded extension from $\frak{R}_{p,cb}^I \otp
\frak{R}_{p,cb}^I$ into $\frak{R}_{p,cb}^I$. Moreover, this map is
completely contractive. Finally, we complete the proof by a
straightforward approximation argument using Lemma \ref{lemma approx
in Rp}.
\end{proof}

\begin{remark}
We do not know if the space $\frak{R}_{p}^I$ equipped with the usual
matricial product is a Banach algebra. The Banach space analogue of
Proposition \ref{theo tensor By IdSp} is false. It is the reason
which explains that the method does not work for $\frak{R}_{p}^I$.
However, note that if $\frak{M}_{p}^I=\frak{M}_{p,cb}^I$
isometrically we have $\frak{R}_{p}^I=\frak{R}_{p,cb}^I$
isometrically. For $1<p<\infty$, $p\not=2$ the equality
$\frak{M}_{p}^I=\frak{M}_{p,cb}^I$ is a classical open question.

\end{remark}


\section{Schur product}

In this section, we replace the matricial product by the Schur
product. First, it is easy to show the following proposition.

\begin{prop}
Suppose $1 \leq p < \infty$. The Banach space $\frak{R}_{p}^I$
equipped with the Schur product is a commutative Banach algebra.
\end{prop}

\begin{proof}
It suffices to use the equality (\ref{expression norme}) and the
fact that $S_p^I$ equipped with the Schur product is a Banach
algebra (see \cite{BLM} page 225).
\end{proof}

Now we will show the completely bounded analogue of this
proposition. We define the pointwise product
$$
\begin{array}{cccc}
  P:  &  \ell_{1}^I \otp \ell_{1}^I   &  \longrightarrow   &  \ell_{1}^I  \\
    &  e_i \ot e_j   &  \longmapsto       &   \delta_{ij}e_i. \\
\end{array}
$$
This map is well-defined and is completely contractive (see
\cite{BLM} page 211). Then, by tensoring, we obtain a completely
contractive map
\begin{equation}\label{first map}
P \ot P\colon \big(\ell_1^I \otp \ell_1^I\big) \ot_h \big(\ell_1^I
\otp \ell_1^I\big) \to \ell_1^I \ot_h \ell_1^I.
\end{equation}
By \cite[Theorem 6.1]{ER2}, the map
$$
\begin{array}{cccccc}
  \big(\ell_\infty^I\overline{\ot} \ell_\infty^I\big)\ot_{\sigma h}\big(\ell_\infty^I \overline{\ot}\ell_\infty^I \big)    &     \longrightarrow     &    \big(\ell_\infty^I\ot_{\sigma h} \ell_\infty^I\big)\overline{\ot}\big(\ell_\infty^I \ot_{\sigma h}\ell_\infty^I \big)    \\
                 a \ot b \ot c \ot d                    &       \longmapsto       &             a \ot c \ot b \ot d                             \\
\end{array}
$$
is completely contractive. Moreover, by \cite[(5.23)]{ER2}, we have
the following commutative diagram
$$
\xymatrix @R=2cm @C=3cm{
  \big(\ell_\infty^I\overline{\ot} \ell_\infty^I\big)\ot_{\sigma h}\big(\ell_\infty^I \overline{\ot}\ell_\infty^I \big) \ar[r]^{}   & \big(\ell_\infty^I\ot_{\sigma h} \ell_\infty^I\big)\overline{\ot}\big(\ell_\infty^I \ot_{\sigma h}\ell_\infty^I \big)\\
    \big(\ell_\infty^I\overline{\ot} \ell_\infty^I\big)\ot_{eh}\big(\ell_\infty^I \overline{\ot}\ell_\infty^I \big) \ar[r]_{} \ar@{^{(}->}[u]_{}& \ar@{^{(}->}[u]_{}\big(\ell_\infty^I\ot_{e h} \ell_\infty^I\big)\overline{\ot}\big(\ell_\infty^I \ot_{e h}\ell_\infty^I \big).\\
  }
$$
By \cite[Theorem 4.2]{ER2}, \cite[Theorem 5.3]{ER2} and by duality,
we deduce that the map
$$
\begin{array}{cccccc}
  \big(\ell_1^I \ot_h  \ell_1^I\big)   \otp  \big(\ell_1^I   \ot_h  \ell_1^I\big)     &     \longrightarrow     &    \big(\ell_1^I \otp  \ell_1^I\big)   \ot_h  \big(\ell_1^I   \otp  \ell_1^I\big)    \\
                 a \ot b \ot c \ot d                    &       \longmapsto       &             a \ot c \ot b \ot d                             \\
\end{array}
$$
is well-defined and completely contractive. Composing this map and
(\ref{first map}), we deduce a completely contractive map
$$
\begin{array}{cccccc}
  \big(\ell_1^I \ot_h  \ell_1^I\big)   \otp  \big(\ell_1^I   \ot_h  \ell_1^I\big)     &     \longrightarrow     &       \ell_1^I \ot_h   \ell_1^I       \\
                 a \ot b \ot c \ot d                    &       \longmapsto       &     P(a \ot c)\ot P(b \ot d).                           \\
\end{array}
$$
With the identification $\frak{R}_{1,cb}^I=\ell_1^I \ot_{h}
\ell_1^I$, we obtain a completely contractive map
$$
\begin{array}{cccc}
    & \frak{R}_{1,cb}^I \otp \frak{R}_{1,cb}^I  &  \longrightarrow   & \frak{R}_{1,cb}^I  \\
    &  A \ot B   &  \longmapsto       &  A*B.  \\
\end{array}
$$
This means that $\frak{R}_{1,cb}^I$ equipped with the Schur product
is a completely contractive Banach algebra. Now, recall that we have
$\frak{R}_{2,cb}^I=\ell_1^I \otp \ell_1^I$ completely isometrically.
Then, by a similar argument, $\frak{R}_{2,cb}^I$ equipped with the
Schur product is also a completely contractive Banach algebra. We
will use a strategy similar to that used in the proof of Theorem
\ref{th principal} for other values of $p$.

We start by defining the Schur multiplier $M_{E}\colon S_p^I(S_p^I)
\to S_p^I(S_p^I)$ associated with the matrix
$E=[\delta_{rt}\delta_{su}]_{(t,r),(u,s) \in I\times I}$ of $
\Mat_{I\times I}$. It is not difficult to see that $M_E$ is a
completely positive contraction. Note that, for all $i,j,k,l\in I$,
we have
\begin{align*}
 M_{E}(e_{ij}\ot e_{kl})
   &= \Big([\delta_{rt}\delta_{su}]_{(t,r),(u,s) \in I\times I}\Big)*\Big([\delta_{it}\delta_{ju}\delta_{kr}\delta_{ls}]_{(t,r),(u,s) \in I\times I}\Big) \\
   &=  \delta_{ik}\delta_{jl}[\delta_{it}\delta_{ju}\delta_{kr}\delta_{ls}]_{(t,r),(u,s) \in I\times I}\\
   &= \delta_{ik}\delta_{jl}e_{ij}\ot e_{kl}.
\end{align*}
Now, we define the linear map
$$
\begin{array}{cccc}
  \eta:  &   \Mat_I  &  \longrightarrow   &  \Mat_{I\times I}  \\
         &   A  &  \longmapsto       & [a_{rs} \delta_{rt} \delta_{su}]_{(t,r),(u,s) \in I\times I}.   \\
\end{array}
$$

\begin{prop}
\label{the diag commutes2} Let $1 \leq p \leq \infty$. Suppose that
$M_A\colon S_p^I \to S_p^I$ is a completely bounded Schur multiplier
on $S_p^I$ associated with a matrix $A$. Then the map $M_E(M_{A} \ot
Id_{S_p^I})M_E$ is a bounded Schur multiplier on $S_p^I(S_p^I)$. Its
associated matrix is $\eta(A)$.
\end{prop}

\begin{proof}
If $i,j,k,l\in I$ and $M_A\in \frak{M}_{p,cb}^I$, we have
\begin{align}
\label{calcul Meta}
 M_{\eta(A)}(e_{ij}\ot e_{kl})
   &= \Big([a_{rs}\delta_{rt} \delta_{su}]_{(t,r),(u,s)  \in I\times I}\Big)*\Big([\delta_{it}\delta_{ju}\delta_{kr}\delta_{ls}]_{(t,r),(u,s)  \in I\times I}\Big)\nonumber\\
   &= \delta_{ik} \delta_{jl}a_{ij}[\delta_{it}\delta_{ju}\delta_{kr}\delta_{ls}]_{(t,r),(u,s)  \in I\times I}\nonumber\\
   &= \delta_{ik} \delta_{jl} a_{ij} e_{ij} \ot e_{kl}
\end{align}
and
\begin{align*}
   M_E(M_{A} \ot Id_{S_p^I})M_E(e_{ij}\ot e_{kl})
   &=  \delta_{ik} \delta_{jl}  M_{E}\big(M_{A} \ot Id_{S_p^I}\big)(e_{ij}\ot e_{kl})\\
   &=  \delta_{ik} \delta_{jl} a_{ij} e_{ij} \ot e_{kl}.\qedhere
\end{align*}
\end{proof}

\begin{thm}
\label{Rp is an algebra for Schur product} Suppose $1 \leq p <
\infty$. The space $\frak{R}_{p,cb}^I$ equipped with the Schur
product is a commutative completely contractive Banach algebra.
\end{thm}

\begin{proof}
We have already seen that it suffices to prove the theorem with $1<
p< \infty$. If $M_A \in \frak{M}_{p,cb}^I$, by Proposition \ref{the
diag commutes2}, we have the following commutative diagram
$$
\xymatrix @R=2cm @C=3cm{
    S_{p}^I\big(S_{p}^I\big) \ar[r]^{M_{\eta(A)}} \ar[d]_{M_{E}}  & S_{p}^I\big(S_{p}^I\big)  \\
    S_{p}^I\big(S_{p}^I\big) \ar[r]_{M_{A}\ot Id_{S_{p}^I}} & S_{p}^I\big(S_{p}^I\big).\ar[u]_{M_{E}}
}
$$
We have already seen that the map $M_A \mapsto M_A \ot Id_{S_p^I}$
is completely contractive from $\frak{M}_{p,cb}^I $ into
$\frak{M}_{p,cb}^{I \times I} $ and w*-continuous. Since
$S_p^I\big(S_p^I\big)$ is reflexive, by Lemma \ref{lemmaVW} and by
composition, the map $M_A \mapsto M_{\eta(A)}$ from
$\frak{M}_{p,cb}^I$ into $\frak{M}_{p,cb}^{I\times I}$ is a complete
contraction and is w*-continuous.

We denote by $\eta_{*}\colon \frak{R}_{p,cb}^{I\times I} \to
\frak{R}_{p,cb}^I$ its preadjoint. Now, by Lemma \ref{lemma approx
in Rp}, we have for all $i,j\in I$ and for all matrices $A,B$ of
$\Mat_I^{\fin}$
\begin{align*}
   \big[\eta_{*}(A\ot B)\big]_{ij}
   &=  \Big\langle M_{e_{ij}},\eta_{*}(A \ot B) \Big\rangle_{\frak{M}_{p,cb}^I,\frak{R}_{p,cb}^I} \\
   &=  \Big\langle M_{\eta(e_{ij})},A \ot B \Big\rangle_{\frak{M}_{p,cb}^{I\times I},\frak{R}_{p,cb}^{I\times I}} \\
   &= \Big\langle M_{[\delta_{ir}\delta_{js}\delta_{rt}\delta_{su}]_{(t,r),(u,s) \in I \times I}},[a_{tu}b_{rs}]_{(t,r),(u,s) \in I \times I} \Big\rangle_{\frak{M}_{p,cb}^{I \times I},\frak{R}_{p,cb}^{I \times I}}\\
   &= a_{ij}b_{ij}\\
   &= [A*B]_{ij}.
\end{align*}
Thus we conclude that if $A,B \in \Mat_I^{\fin}$ we have
$\eta_{*}(A\ot B)=A*B$. By Proposition 3.1 and by density of
$\Mat_I^{\fin} \ot \Mat_I^{\fin}$ in $\frak{R}_{p,cb}^I \otp
\frak{R}_{p,cb}^I$, we deduce that the map
$$
\begin{array}{cccccc}
 \Mat_I^{\fin} \ot \Mat_I^{\fin}     &     \longrightarrow     &  \frak{R}_{p,cb}^{I \times I}   &  \xra{\eta_{*}}& \frak{R}_{p,cb}^I \\
  A \ot B                        &       \longmapsto       &         A \ot B                   &     \longmapsto  &       A*B       \\
\end{array}
$$
admits a unique bounded extension from $ \frak{R}_{p,cb}^I \otp
\frak{R}_{p,cb}^I$ into  $\frak{R}_{p,cb}^I$. Moreover, this map is
completely contractive. Finally, we complete the proof by a
straightforward approximation argument with Lemma \ref{lemma approx
in Rp}.
\end{proof}

Now, we will give a more simple proof of this theorem. It is easy to
see that $\eta$ induces a completely isometric map $\eta\colon
S_{p}^I \to S_{p}^I\big(S_{p}^I\big)$. Moreover, by the computation
(\ref{calcul Meta}), its range is clearly 1-completely complemented
by $M_E\colon S_{p}^I\big(S_{p}^I\big) \to
S_{p}^I\big(S_{p}^I\big)$. We denote by $\eta^{-1}\colon
\eta\big(S_{p}^I\big(S_{p}^I\big)\big)\to S_{p}^I$ the inverse map
of $\eta$. For all $B\in \eta\big(S_{p}^I(S_{p}^I)\big)$, we have
$\eta^{-1}(B)=\big[b_{(r,r),(s,s)}\big]_{r,s \in I}$. Finally, for
all $i,j,k,l \in I$ we have
\begin{align*}
  \eta M_A\eta^{-1}M_E(e_{ij}\ot e_{kl})
   &= \delta_{ik}\delta_{jl}\eta M_A\eta^{-1}(e_{ij}\ot e_{kl}) \\
   &=  \delta_{ik}\delta_{jl}\eta M_A \eta^{-1}\Big([\delta_{it}\delta_{ju}\delta_{kr}\delta_{ls}]_{(t,r),(u,s) \in I \times I}\Big)\\
   &=  \delta_{ik}\delta_{jl}\eta M_A\Big([\delta_{ir}\delta_{js}\delta_{kr}\delta_{ls}]_{r,s \in I}\Big)\\
   &=  \delta_{ik}\delta_{jl}a_{ij}\eta \Big([\delta_{ir}\delta_{js}\delta_{kr}\delta_{ls}]_{r,s \in I}\Big)\\
   &= \delta_{ik} \delta_{jl} a_{ij} e_{ij} \ot e_{kl}\\
   &= M_{\eta(A)}(e_{ij}\ot e_{kl})
\end{align*}
where we have used the computation (\ref{calcul Meta}) in the last
equality.

Hence we have the following commutative diagram
$$
\xymatrix @R=2cm @C=3cm{
    S_{p}^I\big(S_{p}^I\big) \ar[r]^{M_{\eta(A)}} \ar[d]_{M_E}  & S_{p}^I\big(S_{p}^I\big)  \\
     \eta\big(S_{p}^I(S_{p}^I)\big)\ar[d]_{\eta^{-1}}&\\
    S_{p}^I \ar[r]_{M_{A}} & S_{p}^I. \ar[uu]_{\eta}}
$$
We conclude with an argument similar to that used in the proof of
Theorem \ref{Rp is an algebra for Schur product}.

\section{Isometric multipliers}
The next result is the noncommutative version of a theorem of
Parrott \cite{Par} and Strichartz \cite{Str} which states that every
isometric Fourier multiplier on $L_p(G)$ for $ 1\leq p\leq\infty$,
$p\not=2$, is a scalar multiple of an operator induced by a
translation.

\begin{thm}
Suppose $1 \leq p \leq\infty$, $p\not=2$. An isometric Schur
multiplier on $S_p^I$ is defined by a matrix $[a_ib_j]$ with
$a_i,b_j \in \mathbb{T}$.
\end{thm}

\begin{proof}
Suppose that $M_C$ is an isometric Schur multiplier on  the Banach
space $S_p^I$ defined by a matrix $C$. First, we observe that $M_C$
is onto. Indeed, for all $i,j\in I$, we have
$M_C(e_{ij})=c_{ij}e_{ij}$. Then $c_{ij}\not=0$ since $M_C$ is
one-to-one. Consequently $e_{ij}$ belongs to the range of $M_{C}$.
By density, we conclude that $M_C$ is onto.

Now we use the theorem of Arazy \cite{Ara} which describes the onto
isometries on $S_p^I$. Then there exists two unitaries $U=[u_{ij}]$
and $V=[v_{ij}]$ of $B\big(\ell_2^I\big)$ satisfying for all $A \in
S_p^I$
$$
C*A=UAV \ \ \ \text{or} \ \ \ C*A=UA^{T}V.
$$

Examine the first case, we have for all $k,l\in I$
$$
Ue_{kl}V=C*e_{kl}.
$$
Hence, for all $i,j\in I$, we have the equality
$$
[Ue_{kl}V]_{ij}=[C*e_{kl}]_{ij}.
$$
Since
$$
[Ue_{kl}V]_{ij}=u_{ik}v_{lj}
$$
we have
$$
u_{ik}v_{lj}=\left\{
\begin{array}{cl}
  c_{kl}  & \text{if}\ i=k \ \text{and}\  j=l        \\
      0   &  \text{if}   \ i\not=k\ \text{or if}\ j\not=l.    \\
\end{array}\right.
$$
Then $u_{kk}v_{ll}=c_{kl}$. Each $c_{kl}$ is non null since the
image of each $e_{kl}$ by the map $M_C$ cannot be null. Then, for
all $k$ and all $l$, we have $u_{kk}\not=0$ and $v_{ll}\not=0$. And
for $i\not=k$, we have $u_{ik}v_{ll}=0$. Then if $i\not=k$, we have
$u_{ik}=0$. Now if $j\not=l$, we have $u_{kk}v_{lj}=0$. Then if
$j\not=l$, we have $v_{lj}=0$. Finally, for all $i,j\in I$, we
define the complex numbers $a_i=u_{ii}$ and $b_j=v_{jj}$. Since the
diagonal matrices $U$ and $V$ are unitaries, we have $a_{i},b_{j}\in
\mathbb{T}$. Thus we have the required form.

Examine the second case. We have for all $k,l\in I$
$$
Ue_{lk}V=C*e_{kl}.
$$
We deduce that, for all $i,j,k,l \in I$, we have
$$
[Ue_{lk}V]_{ij}=[C*e_{kl}]_{ij}.
$$
Since
$$
[Ue_{lk}V]_{ij}= u_{il}v_{kj}
$$
we obtain $u_{kl}v_{kl}=c_{kl}$ and $u_{il}v_{kj}=0$ if $i\not=k$ or
if $j\not=l$. Each $c_{kl}$ is non null since the image of each
$e_{kl}$ by the map $M_C$ cannot be null. Then for all $k,l$ we have
$u_{kl}\not=0$ and $v_{kl}\not=0$. Thus the second case is absurd
$\big($if $\card(I)>1\big)$.

The converse is straightforward.
\end{proof}

\begin{remark}
It is easy to see that an isometric Schur multiplier on $S_2^I$ is
defined by a matrix $[a_{ij}]$ with $a_{ij}\in \mathbb{T}$.
\end{remark}

The next result is the noncommutative version of a theorem of
Fig\`a-Talamanca \cite{Fig} which states that the space of bounded
Fourier multipliers is the closure in the weak operator topology of
the span of translation operators.

\pagebreak[2]
\begin{thm}
Suppose $1 \leq p < \infty $.
\begin{enumerate}
  \item The space $\frak{M}_{p,cb}^I$ of completely bounded Schur multipliers on $S_p^I$ is the closure of the
span of isometric Schur multipliers in the weak* topology and in the
weak operator topology.
  \item The space $\frak{M}_p^I$ of bounded Schur multipliers on $S_p^I$ is the closure of the
span of isometric Schur multipliers in the weak* topology and in the
weak operator topology.
\end{enumerate}
\end{thm}

\begin{proof}
We will only prove the part 1. The proof of the part 2 is similar.

It is easy to see that an isometric Schur multiplier on $S_{p}^I$ is
completely isometric. This fact allows us to consider the span of
isometric Schur multipliers in $\frak{M}_{p,cb}^I$. Let $C$ be a
matrix of $\frak{R}_{p,cb}^I$. Suppose that $C$ belongs to the
orthogonal of the set of isometric Schur multipliers. Thus, we have
for any isometric multiplier $M_{[a_ib_j]}$ (with $a_i, b_j \in
\mathbb{T}$)
\begin{align*}
0 &=  \Big\langle
M_{[a_ib_j]},C\Big\rangle_{\frak{M}_{p,cb}^I,\frak{R}_{p,cb}^I}\\
  &=  \lim_{J}\sum_{i,j\in J} a_i b_j c_{ij}.
\end{align*}
Let $i_0,j_0$ be elements of $I$. Now, we choose the $a_i$'s,
$b_j$'s, $a'_i$'s and $b'_j$'s such that $a_i=b_j=1$ for all $i,j\in
I$, $a'_i=-1$ if $i \not = i_0$, $a'_{i_0}=1$, $b'_j=-1$ if $j \not
= j_0$ and $b'_{j_0}=1$. Then, we have
\begin{align*}
0  &=  \lim_{J}\sum_{i,j\in J}a_{i}b_{j}c_{ij}+\lim_{J}\sum_{i,j\in J}a_{i}b'_{j}c_{ij}+\lim_{J}\sum_{i,j\in J}a'_{i}b_{j}c_{ij}+\lim_{J}\sum_{i,j\in J}a'_{i}b'_{j}c_{ij} \\
   &= \lim_{J}\sum_{i,j\in J} (a_i+a'_i)(b_j+b'_j)c_{ij} \\
   &=  4c_{i_0j_0}.
\end{align*}
Hence $c_{i_0j_0}=0$. It follows that $C=0$. Then, we deduce that
the space $\frak{M}_{p,cb}^I$ of completely bounded Schur
multipliers is the closure of the span of isometric Schur
multipliers in the weak* topology. Moreover, this topology is more
finer that the weak operator topology. Thus, we have proved the
theorem.
\end{proof}


\subsection*{Acknowledgment}
The author is grateful to his thesis adviser Christian Le Merdy for
support and advice. He also thanks Jean-Christophe Bourin, Pierre
Fima and Eric Ricard for fruitful discussions and the anonymous
referee for useful comments.


\begin{thebibliography}{1}

\bibitem{Ara}
Arazy, J.: The isometries of $C_p$. Israel J. Math. \textbf{22},
247--256 (1975)

\bibitem{BeL}
Bergh, J., L{\"o}fstr{\"o}m, J.: Interpolation spaces.
Springer-Verlag, Berlin (1976)

\bibitem{BLM}
Blecher, D., Le Merdy, C.: Operator algebras and their modules-an
operator space approach. Oxford University Press (2004)

\bibitem{BlP}
Blecher, D., Paulsen, V.: Tensor products of operator spaces. J.
Funct. Anal. \textbf{99}, 262--292 (1991)

\bibitem{Daw1}
Daws, M.: $p$-Operator Spaces and Fig\`a-Talamanca-Herz algebras. J.
Operator Theory \textbf{63}, 47--83 (2010)

\bibitem{Daw2}
Daws, M.: Representing multipliers of the Fourier algebra on
non-commutative $L^p$ spaces (2009). arXiv:0906.5128v2[math.FA]

\bibitem{ER1}
Effros E., Ruan, Z.-J.: Operator spaces. Oxford University Press
(2000)

\bibitem{ER2}
Effros, E., Ruan Z.-J.: Operator space tensor products and Hopf
convolution algebras. J. Operator Theory \textbf{50}, 131--156
(2003)

\bibitem{Eym}
Eymard, P.: Algèbres $A_p$ et convoluteurs de $L^p$. In S\'eminaire
Bourbaki, vol. 1969/1970, Expos\'es 364--381. Springer-Verlag (1971)

\bibitem{Fig}
Fig\`a-Talamanca, A.: Translation invariant operators in $L^p$. Duke
math. \textbf{32}, 495--501 (1965)

\bibitem{Har}
Harcharras, A.: Fourier analysis, Schur multipliers on $S^p$ and
noncommutative $\Lambda(p)$-sets. Studia math. \textbf{137},
203--260 (1999)

\bibitem{Her}
Herz, C.: The theory of $p$-spaces with an application to
convolution Operators. Trans. Amer. Math. Soc. \textbf{154}, 69--82
(1971)

\bibitem{Hla}
Hladnik, M.: Compact Schur multipliers. Proc. Amer. Math. Soc.
\textbf{128}, 2585--2591 (2000)

\bibitem{Lam}
Lambert A., Neufang M., Runde, V.: Operator space structure and
amenability for Fig\`a-Talamanca-Herz algebras. J. Funct. Anal.
\textbf{211}, 245--269 (2004)

\bibitem{Lar}
Larsen, R.: An introduction to the theory of multipliers.
Springer-Verlag (1971)

\bibitem{Neu}
Neuwirth, S.: Cycles and 1-unconditional matrices. Proc. London
Math. Soc. \textbf{93}, 761--790 (2006)

\bibitem{Par}
Parott, S. K.: Isometric multipliers. Pacific J. Math. \textbf{25},
159--166 (1968)

\bibitem{Pau}
Paulsen, V.: Completely bounded maps and operator algebras.
Cambridge University Press (2002)

\bibitem{Pie}
Pier, J.-P.: Amenable Locally Compact Groups. Wiley-Interscience
(1984)

\bibitem{Pis1}
Pisier, G.: The operator {H}ilbert space {${\rm OH}$}, complex
interpolation and tensor norms. Mem. Amer. Math. Soc. \textbf{122}
(1996)

\bibitem{Pis2}
Pisier, G.: Non-commutative vector valued {$L\sb p$}-spaces and
completely {$p$}-summing maps. Ast\'erisque \textbf{247} (1998)

\bibitem{Pis3}
Pisier, G.: Similarity problems and completely bounded maps. Lecture
Notes in Mathematics \textbf{1618}. Expanded edition,
Springer-Verlag (2001)

\bibitem{Pis4}
Pisier, G.: Introduction to operator space theory. Cambridge
University Press (2003)

\bibitem{Run1}
Runde, V.: Operator Fig\`a-Talamanca-Herz algebras. Studia Math.
\textbf{155}, 153--170 (2003)

\bibitem{Run2}
Runde, V.: Representations of locally compact groups on
$QSL_p$-spaces and a $p$-analog of the Fourier-Stieltjes algebra.
Pacific J. Math. \textbf{221}, 379--397 (2005)

\bibitem{Spr}
Spronk, N.: Measurable Schur multipliers and completely bounded
multipliers of the Fourier algebras. Proc. London Math. Soc.
\textbf{89}, 161--192 (2004)

\bibitem{Str}
Strichartz, R. S.: Isomorphisms of group algebras. Proc. Amer. Math.
Soc. \textbf{17}, 858--862 (1966)

\bibitem{Xu}
Xu, Q.: Interpolation of Schur multiplier spaces. Math. Z.
\textbf{235}, 707--715 (2000)

\end{thebibliography}
\end{document}